\documentclass{amsart}

\usepackage{amssymb}
\usepackage{amsmath}
\usepackage{amsthm}
\usepackage{xcolor}
\usepackage{comment}
\usepackage{tikz}
\usepackage{xypic}

\newtheorem{thm}{Theorem}[section]
\newtheorem{lem}[thm]{Lemma}
\newtheorem{prop}[thm]{Proposition}
\newtheorem{cor}[thm]{Corollary}

\newtheorem*{dualramsey}{Dual Ramsey Theorem}

\newtheorem*{carlsonsimpsonlemma}{Carlson-Simpson Lemma}

\theoremstyle{definition}
\newtheorem{defn}[thm]{Definition}

\newtheorem{question}[thm]{Question}

\theoremstyle{remark}

\newtheorem*{claim}{Claim}

\newcommand{\system}[1]{\mbox{\fontfamily{cmss}\fontshape{n}\fontseries{m}%
    \selectfont#1}}
\newcommand{\RCA}{\system{RCA}\ensuremath{_0}}
\newcommand{\WKL}{\system{WKL}\ensuremath{_0}}
\newcommand{\RT}{\system{RT}}
\newcommand{\ACA}{\system{ACA}\ensuremath{_0}}
\newcommand{\ATR}{\system{ATR}\ensuremath{_0}}

\newcommand{\COH}{\system{COH}}

\newcommand{\DRT}{\system{DRT}}
\newcommand{\BoDRT}{\system{Borel-DRT}}
\newcommand{\BaDRT}{\system{Baire-DRT}}
\newcommand{\CDRT}{\system{CDRT}}
\newcommand{\cDRT}{\system{CDRT}}
\newcommand{\SRT}{\system{SRT}}

\newcommand{\CSL}{\system{CSL}}
\newcommand{\ODRT}{\system{ODRT}}
\newcommand{\OVW}{\system{OVW}}
\newcommand{\VW}{\system{VW}}
\newcommand{\HT}{\system{HT}}

\newcommand{\bdefn}{\begin{defn}}
\newcommand{\edefn}{\end{defn}}
\newcommand{\bthm}{\begin{thm}}
\newcommand{\ethm}{\end{thm}}
\newcommand{\bitem}{\begin{itemize}}
\newcommand{\eitem}{\end{itemize}}
\newcommand{\bpf}{\begin{proof}}
\newcommand{\epf}{\end{proof}}

\newcommand{\restrict}{\upharpoonright}

\newcommand{\kO}{\mathcal{O}}

\newcommand{\fin}{\textup{fin}}
\newcommand{\concat}{^\smallfrown}

\newcommand{\KO}{\mathcal{O}}

\newcommand{\block}[2]{{#1}^{-1}(#2)}

\title{Effectiveness for the Dual Ramsey Theorem}
\author{Damir D.~Dzhafarov}
\author{Stephen Flood}
\author{Reed Solomon}
\author{Linda Westrick}

\begin{document}

\begin{abstract}We analyze the Dual Ramsey Theorem 
for $k$ partitions and $\ell$ colors ($\DRT^k_\ell$) in the context of 
reverse math, effective analysis, and strong reductions.  Over $\RCA$, the Dual Ramsey Theorem 
stated for Baire colorings $\BaDRT^k_\ell$ is equivalent to the statement for clopen colorings 
$\ODRT^k_\ell$ and to a purely combinatorial theorem $\cDRT^k_\ell$. 

When the theorem is stated for Borel colorings and $k\geq 3$,  
the resulting principles are essentially relativizations of $\cDRT^k_\ell$.  
For each $\alpha$, there is a computable Borel code for a 
$\Delta^0_\alpha$ coloring such that any partition homogeneous for it
computes $\emptyset^{(\alpha)}$ or $\emptyset^{(\alpha-1)}$ depending 
on whether $\alpha$ is infinite or finite.  

For $k=2$, we present partial results giving bounds on the effective content of the principle.  
A weaker version for $\Delta^0_n$ reduced colorings is equivalent to 
$\mathsf{D}^n_2$ over $\RCA+\mathsf{I}\Sigma^0_{n-1}$ and in the sense 
of strong Weihrauch reductions.
\end{abstract}

\maketitle

\footnotetext{Dzhafarov was supported in part by NSF Grant DMS-1400267. The authors thank Jos\'{e} Mijares and Ludovic Patey for useful comments and 
discussions during the preparation of this paper.  They also thank the anonymous
referee for many helpful comments which have improved the presentation.}

\section{Introduction}
\label{sec:intro}

This paper concerns the reverse mathematical and computational strength of 
variations of the Dual Ramsey Theorem.  For $k \leq \omega$, 
let $(\omega)^k$ denote the
set of all partitions of $\omega$ into exactly $k$ pieces.
Such a partition can be represented as a surjective function from 
$\omega$ to $k$.  Thus $(\omega)^k$ inherits a natural topology
by considering it as a subset 
of $k^\omega$.

\begin{dualramsey}[\cite{cs}, \cite{pv}]
For any $k,\ell < \omega$,
suppose we have a coloring $(\omega)^k = \cup_{i<\ell} C_i$.  If for each
$i<\ell$, $C_i$
has the property of Baire, 
then there is a partition $p \in (\omega)^\omega$ 
such that any coarsening of $p$ down to 
exactly $k$ pieces has the same color.
\end{dualramsey} 

The reason that this theorem is dual to the original Ramsey's Theorem
concerns what objects are being colored.  In the original Ramsey's theorem, 
we color the $k$-element subsets of $\omega$, which correspond to 
injective functions from $k$ to $\omega$.  In the Dual Ramsey Theorem, we color 
surjective functions from $\omega$ to $k$.

A straightforward choice argument shows that the Dual Ramsey Theorem fails if no 
regularity conditions on the $C_i$ are assumed.
The theorem was first
proved for Borel colorings by Carlson and Simpson \cite{cs}, and 
extended to colorings with the Baire property 
by Pr\"{o}mel and Voigt \cite{pv}.  From the perspective of reverse mathematics 
or computational mathematics, the variation in hypothesis gives us two theorems
to consider.  We call them the
\emph{Borel Dual Ramsey Theorem} and the \emph{Baire Dual Ramsey Theorem}
respectively.

Carlson and Simpson asked for a recursion-theoretic analysis of the 
Borel Dual Ramsey Theorem.  In order to answer this, it is necessary to 
choose a method for encoding the coloring, and one must 
consider the potential effects of a topologically intricate coloring.  Previous 
work side-stepped these issues by restricting attention to open colorings only 
\cite{ms} or by focusing attention only on the main combinatorial lemma 
which Carlson and Simpson used in their proof, and on its 
variable word variants \cite{ms, e, LiuMoninPatey}.  

From the work of \cite{ms}, we know that over $\RCA$, 
$\ODRT^{k}_\ell$ implies $\RT^{k-1}_\ell$, where $\ODRT^k_\ell$ is the 
restriction of the Borel Dual Ramsey Theorem to open colorings only,
and $\RT$ is the usual Ramsey's Theorem.
This provides a lower bound on the strength of the Borel Dual Ramsey 
Theorem.  Conversely, in unpublished work Slaman has shown that 
the Borel Dual Ramsey Theorem follows from $\Pi^1_1\text{-}\mathsf{CA}_0$ \cite{slaman}.
No direct implication is known between the Dual Ramsey Theorems
and the variable word theorems, because the Dual
Ramsey Theorem does not require the ``words'' in its solution to be finite
(and by Proposition \ref{prop:2.13}, it cannot require this), 
while the proof of the Dual Ramsey Theorem
from the variable word theorems uses infinitely many sequential applications
of the latter (Theorem \ref{thm.3.31}).
Overall, this leaves a rather large gap, and we do not close it.  However, we do 
provide significant clarification of the key difficulties.
In particular, for the first time
we directly tackle the topological aspect of the Borel version of the 
theorem.  

\subsection{Combinatorial core of the Borel Dual Ramsey Theorem}

Since the Borel version follows from the Baire version plus the additional
principle ``Every Borel set has the property of Baire'', our first step is to understand
the Baire version. 
 
To be clear, an instance of the Baire Dual Ramsey Theorem is 
a sequence of pairwise
disjoint open sets $O_0,\dots,O_{\ell-1}$ whose union is dense
 in $(\omega)^k$,
and a sequence of dense open sets $\{D_n\}_{n\in\omega}$.
Such an instance simultaneously represents all colorings 
\mbox{$(\omega)^k = \cup_{i<\ell} C_i$}
for which the symmetric difference
$C_i \Delta O_i$ is disjoint from $\cap_n D_n$.  
There may be uncountably many such colorings,
because no condition is placed on how $2^\omega \setminus \cap_n D_n$
 is colored.  Any solution $p \in (\omega)^\omega$ to the Baire version must have 
$(p)^k \subseteq \cap_n D_n$.

In Section \ref{sec:2.1} 
we define a purely combinatorial principle $\CDRT^k_\ell$, which
precisely captures the strength of the Baire version.  In the following,
if $p \in (\omega)^\omega$ and $k\leq \omega$, let $(p)^k$ denote the
set of coarsenings of $p$ into exactly $k$ pieces.  Recalling that we 
consider $p$ as a surjective function $p:\omega\rightarrow k$, let 
$$p^\ast := p \restrict \min p^{-1}(k-1).$$
In other words, $p^\ast$ is a string on alphabet $k-1$,
it tells us by its length what is the smallest element 
 of $p$'s last block, and it tells us how $p$ partitions the finitely many smaller elements
 into its first $k-1$ blocks. 
 Let $(<\omega)^{k-1} = \{p^\ast : p \in (\omega)^k\}$.

\begin{thm}\label{thm:BaireOC} Let $k,\ell < \omega$.  Over $\RCA$, the following are equivalent.
\begin{enumerate}
\item The Baire Dual Ramsey Theorem for $k$ partitions and $\ell$ colors.
\item $\ODRT^k_\ell$
\item $\CDRT^k_\ell$, which states: for every $c : (<\omega)^{k-1} \rightarrow \ell$, 
there is a $p\in(\omega)^\omega$ and a color $i<\ell$ 
such that for every $x \in (p)^{k}$, $c(x^\ast) = i$.
\end{enumerate}
\end{thm}

Thus we have reduced the Baire version of the theorem to a purely combinatorial 
statement.  The proof of the equivalence
is essentially an effectivization of \cite{pv}.

Aside from the results in \cite{ms}, 
the strengths of the $\CDRT^k_\ell$ statements are wide open. 
We include one more result, which was known to Simpson (see \cite[page 268]{cs})
and subsequently rediscovered by Patey  \cite{patey}:
a proof of one case of the Carlson-Simpson Lemma
 from Hindman's Theorem.  
With minor modifications, we adapt this
proof in Section \ref{subsec:Hindman} to show that Hindman's Theorem for $\ell$ colorings 
implies the stronger $\CDRT^3_{\ell}$. See Figure \ref{fig:1.1}
for a summary of what is known about the combinatorial core of the Dual 
Ramsey Theorem.

We close Section \ref{sec:2} with a self-contained proof of 
$\CDRT^k_{\ell}$ from the Carlson-Simpson Lemma (Theorem \ref{thm.3.31}).  
In our proof, the only non-constructive steps are 
$\omega \cdot (k-2)$ nested applications of the Carlson-Simpson Lemma.
  
The earliest claim we are aware of for a proof of $\CDRT^k_{\ell}$ is
in \cite{pv}, where a generalization of $\CDRT^k_{\ell}$ called \emph{Theorem A} 
is attributed to a preprint of 
Voigt titled ``Parameter words, trees and vector spaces''. However, as far as we can tell, this paper never appeared.  Another proof of $\CDRT^k_{\ell}$ can be 
found in \cite{TodorcevicRamsey}, but as a corollary of a larger theory.

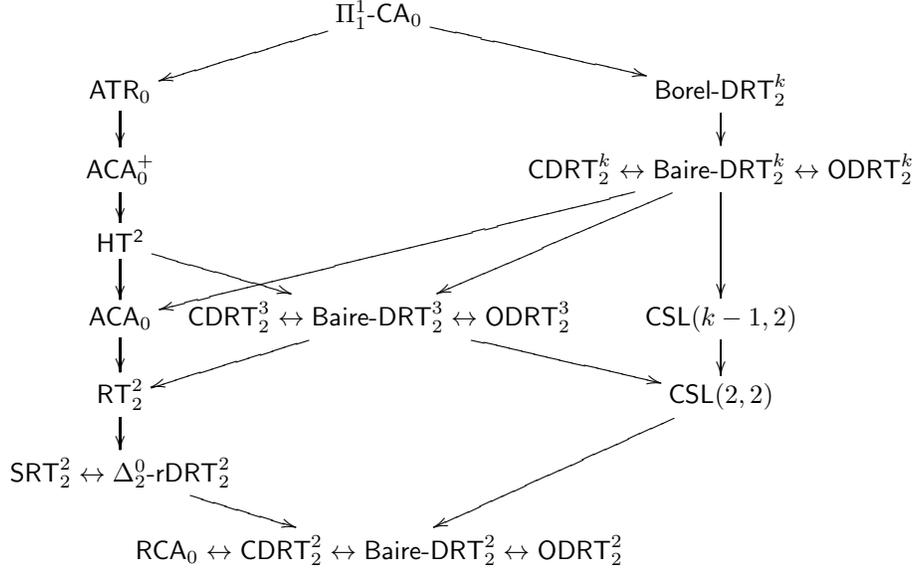
\begin{figure}\label{fig:1.1}\caption{Implications over $\RCA$ between variants of the Dual Ramsey Theorem considered in this paper and some related principles. 
The parameter $k \geq 4$ is arbitrary.}
\[
\xymatrix @C=-3.5pc@R=1pc{
& \Pi^1_1\text{-}\mathsf{CA}_0 \ar[dr] \ar[dl]\\
\mathsf{ATR}_0 \ar[d] & & \mathsf{Borel}\text{-}\mathsf{DRT}^k_2 \ar[d]\\
\mathsf{ACA}^+_0 \ar[d] & & \mathsf{CDRT}^k_2 \leftrightarrow \mathsf{Baire}\text{-}\mathsf{DRT}^k_2 \leftrightarrow  \mathsf{ODRT}^k_2 \ar[dd] \ar[ddl] \ar[ddll] \\ 
\mathsf{HT} \ar[dr] \ar[d]\\
\mathsf{ACA}_0 \ar[d] & \mathsf{CDRT}^3_2 \leftrightarrow \mathsf{Baire}\text{-}\mathsf{DRT}^3_2 \leftrightarrow  \mathsf{ODRT}^3_2 \ar[dl] \ar[dr] & \mathsf{CSL}(k-1,2) \ar[d]\\
\mathsf{RT}^2_2 \ar[d] & & \mathsf{CSL}(2,2) \ar[ddl] \\
\mathsf{SRT}^2_2 \leftrightarrow \Delta^0_2\text{-}\mathsf{rDRT}^2_2 \ar[dr]\\
& \mathsf{RCA}_0 \leftrightarrow \mathsf{CDRT}^2_2 \leftrightarrow \mathsf{Baire}\text{-}\mathsf{DRT}^2_2 \leftrightarrow  \mathsf{ODRT}^2_2
}
\]
 \end{figure}

\subsection{Computational strength of the Borel Dual Ramsey Theorem}

In Sections \ref{sec:Borel3} and \ref{sec:4}, we consider the Borel Dual 
Ramsey Theorem, or $\BoDRT$, from the perspective of effective 
combinatorics.  The behavior is different depending on the number of 
pieces $k$ in the partition, with the $k\geq 3$ case being addressed in Section 
\ref{sec:Borel3} and the $k=2$ case in Section \ref{sec:4}.

When $k\geq 3$, given a fast-growing function $f$ one can design 
an open, $f$-computable 
coloring such that all of its homogeneous partitions compute 
a function which dominates $f$ (this was already essentially done 
in \cite{ms}).  But if $f$ is hyperarithmetic,
that same coloring has an effective Borel code as a $\Delta^0_\alpha$
set.  Thus by sneaking the computation of $f$ into an 
effective Borel code, we obtain a computable instance of 
$\BoDRT^3_2$.  As a result, $\BoDRT^3_2$ can be informally 
considered as some kind of hyperjump of $\ODRT^3_2$.  Formally,
we have the following in Theorem \ref{thm:hyp}.

\begin{thm}
For every computable ordinal $\alpha>0$ and every $k\geq 3$,
there is a computable Borel code for a $\Delta^0_\alpha$ coloring 
$c:(\omega)^k\rightarrow 2$ such that every infinite 
partition homogeneous for $c$ computes $\emptyset^{(\alpha)}$
if $\alpha$ is infinite, or $\emptyset^{(\alpha-1)}$ if $\alpha$ is finite.
\end{thm}

The preceding theorem gives a coding lower bound on the complexity 
of solutions for $k\geq 3$.  In contrast,
we remark that the best known basis theorem for the 
$k\geq 3$ case is still the following result of Slaman \cite{slaman}:
Every hyperarithmetic instance of the Borel Dual Ramsey Theorem
has a hyperarithmetically low solution. This result can also be 
extracted from our analysis as follows.  Given a Borel coloring,
there is a hyperarithmetic witness that it has the property of Baire.
Use Theorem \ref{thm:BaireOC} to computably reduce this instance
of the Baire Dual Ramsey Theorem to an instance of $\CDRT$.
It is arithmetic to check whether 
a given partition $p\in(\omega)^\omega$ 
is a solution to a given instance $c$ of $\CDRT$.  
Therefore the collection of solutions is non-empty $\Sigma^1_1$.
Applying the Gandy Basis Theorem gives the desired solution.

When $k=2$, it is likewise possible to create effectively Borel instances
which correspond to hyperarithmetically computable open
colorings.  However, there are two important differences with 
the $k=2$ case.  First, $\ODRT^2_\ell$ is computably true.  
As a consequence, when $k=2$ the Borel variant has
a sharper basis theorem.
\begin{thm}Every $\Delta^0_n$ instance of $\BoDRT^2_\ell$ has a 
$\Delta^0_n$ solution.
\end{thm}
 This result follows from the more general Theorem \ref{thm:lifting}.
Note that the $\Delta^0_n$ instance is a subset of $(\omega)^k$ which
could be topologically intricate, while the solution is a single
$\Delta^0_n$ partition $p \in (\omega)^\omega$.

The second difference in the $k=2$ case is that $\CDRT^2_\ell$ 
is Weihrauch equivalent to the infinite pigeonhole principle $\RT^1_\ell$.
(Observe that an instance of $\CDRT^2_\ell$ is essentially 
a coloring of $\omega$.)  This immediately offers lower bounds:
for each $n$,
$D^n_\ell \leq_{\mathrm{sW}} \BoDRT^2_\ell$, where $D^n_\ell$ 
is the problem whose instances are $\Delta^0_n$ colorings 
$c:\omega\rightarrow \ell$ and whose solutions are the infinite 
sets monochromatic for $c$.  The question is whether these 
could possibly be equivalences when $\BoDRT^2_\ell$ is
likewise restricted to $\Delta^0_n$ instances.  We are 
only able to show a partial result in this direction (Theorem \ref{thm:sW}).

\begin{thm}
Let $\Delta^0_n$-$\mathsf{rDRT}^2_2$ be the restriction of $\BoDRT^2_2$
to instances $c$ which are given by $\Delta^0_n$ formulas 
and for which $c$ is \emph{reduced}, meaning that $c(p)$ 
depends only on $p^\ast$ for all $p\in(\omega)^2$.  Then
\begin{enumerate}
\item $\Delta^0_n$-$\mathsf{rDRT}^2_2 \equiv_{\mathrm{sW}} \mathsf{D}^n_2$.
		\item Over $\RCA + \mathsf{I}\Sigma^0_{n-1}$, $\Delta^0_n$-$\mathsf{rDRT}^2_2$ is equivalent to $\mathsf{D}^n_2$.  
\end{enumerate}
\end{thm}

\subsection{Reverse mathematics and Borel sets}

In Section \ref{sec:codes}, we consider problems motivated by
the reverse mathematics 
of the Borel Dual Ramsey Theorem.  We observe that 
the Borel Dual Ramsey Theorem can be obtained by composing 
``Every Borel set has the property of Baire'' (let us call it $\mathsf{BP}$)
with the Baire Dual Ramsey Theorem.  So a natural next step is to 
understand the strength of $\mathsf{BP}$.  We show the following
as a part of Theorem \ref{ATRequiv}.
\begin{thm}\label{thm:sillyATR}
Over $\RCA$, $\ATR$ is equivalent to the following statement.
For every Borel code $B$, there is some point $x$ 
such that $x \in B$ or $x \not \in B$. 
\end{thm}
This result mainly shows that the usual definition of Borel sets,
which is given in \cite{sosa} using $\ATR$ as a base theory,
really does not make sense in the absence of $\ATR$.  This 
provides an obstacle to a satisfactory analysis of $\mathsf{BP}$.
While $\mathsf{BP}$ follows from $\ATR$, (Proposition \ref{prop:BP}), 
in the reversal $\mathsf{BP}$ formally implies $\ATR$ only due to the 
technical reason above.  We leave a deeper analysis 
of $\mathsf{BP}$
and the Borel Dual Ramsey Theorem to future work \cite{ADMSW}.

The proof of Theorem \ref{thm:sillyATR} 
uses a method of effective transfinite recursion, $\mathsf{ETR}$, which is available in $\ACA$ (and possibly in weaker 
systems). Greenberg and Montalb\'{a}n \cite{GM} use 
$\mathsf{ETR}$ to establish equivalences of $\mathsf{ATR}_0$ and claim that $\mathsf{ETR}$ is provable in $\mathsf{RCA}_0$. However, their proof of $\mathsf{ETR}$ 
overlooks an application of $\Sigma^0_1$ transfinite induction, and in general, transfinite induction for $\Sigma^0_1$ formulas does not hold in $\mathsf{RCA}_0$. 
While the main results in \cite{GM} continue to hold because Greenberg and Montalb\'{a}n show the classified theorems imply $\mathsf{ACA}_0$ without reference to 
$\mathsf{ETR}$ (and hence can use $\mathsf{ETR}$ in $\mathsf{ACA}_0$ to complete the equivalence with $\mathsf{ATR}_0$), we have included a proof of 
$\mathsf{ETR}$ in Section \ref{sec:codes} 
to make explicit the use of transfinite induction.

In the final Section \ref{sec:open} we list a number of open questions.

\section{Notation}
\label{sec:notation}

We use $\omega$ to denote the natural numbers, which in subsystems of $\mathsf{Z}_2$ 
is the set $\{ x : x=x \}$, often 
denoted by $\mathbb{N}$ in the literature. Despite this notation, we do not restrict ourselves to $\omega$-models. Second, when we refer to the parameters $k$ and 
$\ell$ in versions of the Dual Ramsey Theorem, we assume $k$ and $\ell$ are arbitrary standard numbers with $k, \ell \geq 2$. By a statement such as 
``$\RCA$ proves $\BoDRT^k_{\ell}$ implies $\BaDRT^k_{\ell}$'', we mean, for all $k, \ell \geq 2$, $\RCA \vdash \BoDRT^k_{\ell} \rightarrow \BaDRT^k_{\ell}$. 
For many results, the quantification over $k$ and $\ell$ can be pulled inside the formal system. However, in some cases, issues of induction arise and 
we wish to set those aside in this work.

For $k \leq \omega$, let $k^{< \omega}$ denote the set of finite strings over $k$ and let $k^\omega$ denote the set of  functions $f:\omega \rightarrow k$. As noted above, 
unless explicitly stated otherwise, we will always assume that $k \geq 2$. 
For $\sigma \in k^{< \omega}$, $|\sigma|$ denotes the length of $\sigma$, and if $|\sigma| > 0$, $\sigma(0), \ldots, \sigma(|\sigma|-1)$ denote the entries of $\sigma$ in order. 
For $p \in k^{\omega}$ and $\sigma \in k^{< \omega}$, we write $\sigma \prec p$ if $\sigma$ is an initial segment of $p$.  Similarly, if $\sigma, \tau \in
k^{< \omega}$, we write $\sigma \preceq \tau$ if $\sigma$ is an initial segment of $\tau$ and $\sigma \prec \tau$ if $\sigma$ is a
proper initial segment of $\tau$. We write $p \restrict n$ to denote the string obtained by restricting 
the domain of $p$ to $n$. The standard (product) topology on $k^{\omega}$ is generated by basic clopen sets of the form
\[
[\sigma] = \{ p \in k^{\omega} : \sigma \prec p \}
\]
for $\sigma \in k^{< \omega}$. 

We use the following notational conventions for partitions.
For $k \leq s \leq \omega$, we use $(s)^k$ to denote the 
set of all partitions of $s$ into exactly $k$ pieces.
The pieces are also called \emph{blocks}.
Each such partition can be viewed
as a surjective function $p:s\rightarrow k$, where the 
blocks are the sets $\block{p}{i}$ for $i<k$.  More than one 
surjective function can describe the same partition,
so we pick a canonical one.
We say that $p:s\rightarrow k$ 
is \emph{ordered} if for each $i<j<k$, 
$\min \block{p}{i} < \min \block{p}{j}$. 
We then more formally define the \emph{$k$-partitions of $s$} as
$$(s)^k = \{ p \in k^s : \text{ $p$ is surjective and ordered}\}.$$
 We also let $(<\omega)^k$ denote $\cup_{r \in \omega} (r)^k$.
 
If $k \leq s \leq t \leq \omega$ and $p \in (t)^s$, then we define
$(p)^k = \{x \circ p  : x \in (s)^k\}$.  In English, if $p$ is a partition of $t$ 
into exactly $s$ pieces, $(p)^k$ is the set of ways to further coarsen 
$t$ down to exactly $k$ pieces, so we call $(p)^k$ the set of 
\emph{$k$-coarsenings} of $p$.

If $(\omega)^k = \cup_{i<\ell} C_i$ and $p \in (\omega)^\omega$ 
with $(p)^k \subseteq C_i$, then we say that $p$ is \emph{homogeneous}
for the color $C_i$.

The set $(\omega)^k$ inherits the subspace topology from
$k^{\omega}$ with basic open sets of the form $[\sigma] \cap (\omega)^k$ for $\sigma \in k^{< \omega}$.   This topology is also natural
from the partition perspective.  For example, 
if we considered a partition instead as an equivalence
relation $R\subseteq \omega\times\omega$, the same topology is also
generated by declaring $\{R : (n,m)\in R\}$ to be clopen for each pair $(n,m) \in \omega\times\omega$.

The space $(\omega)^k$ is not compact since, for example, 
the collection of open sets $[0^n1]$ for $n \geq 1$ cover $(\omega)^2$ but this collection has no finite subcover. However, if $\sigma \in (<\omega)^k$, then 
$[\sigma] \subseteq (\omega)^k$ and $[\sigma]$ is a compact clopen subset of $(\omega)^k$. To generate the topology on $(\omega)^k$, it suffices to restrict to 
the basic clopen sets of the form $[\sigma]$ with $\sigma \in (<\omega)^k$.  Although the
notation $[\sigma]$ is ambiguous about whether the ambient space is $k^{\omega}$ or $(\omega)^k$ (or $\ell^{\omega}$ or $(\omega)^\ell$ for some $\ell > k$), 
the meaning will be clear from context.

We denote the $i$th block of the partition $p$ by $\block{p}{i}$  (we start 
counting the blocks at 0, so the last block of a $k$-partition is indexed 
by $i=k-1$).  We denote the least element of $\block{p}{i}$ by 
$\mu^p(i)$.  If $p \in (\omega)^k$,
 we will often have use for the 
string $p^\ast = p\restrict\mu^p(k-1)$.
We can also apply this notation if $p \in (s)^k$ for any $s \geq k$.

Sometimes it is convenient to consider colorings of $(p)^k$ 
for some $p \in (\omega)^\omega$, and then ask for 
a homogeneous partition $q \in (p)^\omega$.  This is 
not really more general than the case we have been considering,
because a coloring
$(p)^k = \cup_{i<\ell} C_i$ corresponds canonically 
to the coloring of $(\omega)^k$ defined by
\begin{equation}\label{eqn:canonical}
x \in \widehat C_i \iff x \circ p \in C_i.
\end{equation}
In this case any $y\in (\omega)^\omega$ is homogeneous 
for $\{\widehat C_i\}_{i<\ell}$ if and only if $y \circ p$ 
is homogeneous for $\{C_i\}_{i<\ell}$.

\section{The Baire Dual Ramsey Theorem}\label{sec:2}

\subsection{Three versions of the Baire Dual Ramsey Theorem}\label{sec:2.1} 

We formulate three versions of the Baire
Dual Ramsey Theorem in second order arithmetic and show they 
are equivalent over $\RCA$.

Coding colorings or sets with the Baire property in second order arithmetic is complicated by the fact that there are $2^{\mathfrak{c}}$ (where $\mathfrak{c} = 2^{\aleph_0}$) 
many subsets of $(\omega)^k$ or $k^{\omega}$ with the Baire property.  However, 
if we identify colorings which are the same after discarding a meager
set, then there are only 
continuum many with the Baire property.  Specifying only an equivalence class 
of colorings is consistent with how theorems 
which hypothesize the Baire property usually work. 
They start 
by fixing a comeager approximation to the set in question and then 
proceed by working exclusively with this approximation. This classical observation motivates our definition of a code for a Baire coloring.

\begin{defn}[$\RCA$]
\label{defn:reverseopen}
A \textit{code for an open set in} $(\omega)^k$ is a set $O \subseteq
\omega \times (<\omega)^k$. We say that a partition $p \in (\omega)^k$
is \textit{in the open set coded by} $O$ (or just \textit{in} $O$ and
write $p \in O$) if there is a pair $\langle n, \sigma \rangle \in O$ such that $p \in
[\sigma]$.

A \textit{code for an closed set in} $(\omega)^k$ is also a set $V \subseteq
\omega \times (<\omega)^k$. In this case, we say $p \in (\omega)^k$
\textit{is in} $V$ (and write $p \in V$) if for all pairs $\langle n, \sigma \rangle \in V$, $p \not \in
[\sigma]$.
\end{defn}

\begin{defn}[$\RCA$]
An open set $O \subseteq (\omega)^k$ is \textit{dense} if for all $\tau \in (<\omega)^k$, $[\tau] \cap O \neq \emptyset$. That is, for all $\tau$, there is a pair 
$\langle n, \sigma \rangle \in O$ such that $\sigma$ and $\tau$ are comparable as strings.
\end{defn}

\begin{defn}[$\RCA$]
\label{defn:Baire}
A \textit{code for a Baire} $\ell$-\textit{coloring of} $(\omega)^k$ is a sequence of dense open sets $\{D_n\}_{n<\omega}$ together 
with a sequence of pairwise disjoint
open sets $\{O_i\}_{i<\ell}$ such that $\bigcup_{i < \ell} O_i$ 
is dense in $(\omega)^k$. 
\end{defn}

Recall that $\RCA$ suffices to prove the Baire Category Theorem: if $\{ D_n \}_{n < \omega}$ is a sequence of dense open sets, then $\cap_{n < \omega} D_n$ is dense. 
Classically, if a coloring $\cup_{i<\ell} C_i = (\omega)^k$ has the Baire property, then it has a comeager  
approximation given by sequences of open sets $\{O_i\}_{i<\ell}$ and $\{D_n\}_{n<\omega}$ such that 
each $D_n$ is dense and for each $p \in \cap_{n<\omega} D_n$, $p \in C_i$ if and only if $p \in O_i$. 

We abuse terminology and refer to the Baire code as a \textit{Baire} $\ell$-\textit{coloring} of $(\omega)^k$. 
Similarly, an \textit{open} $\ell$-\textit{coloring} 
is a coloring $(\omega)^k = \cup_{i < \ell} O_i$ in which the $O_i$ are open
 and pairwise disjoint.
\begin{defn}\label{def:2.1}
For each (standard) $k, \ell \geq 2$, we define $\BaDRT^k_{\ell}$, $\ODRT^k_{\ell}$ and $\cDRT^k_{\ell}$ in $\RCA$ as follows.
\begin{enumerate}
\item $\BaDRT^k_\ell$: For every Baire $\ell$-coloring $\{O_i\}_{i<\ell}$ and $\{D_n\}_{n<\omega}$ of $(\omega)^k$, there is a partition $p \in
  (\omega)^\omega$ and a color $i < \ell$ such that for all $x \in (p)^k$, $x
  \in O_i \cap \bigcap_n D_n$.
\item $\ODRT^k_\ell$: For every open $\ell$-coloring $(\omega)^k = \cup_{i<\ell} O_i$, there is a partition $p \in (\omega)^\omega$ and
  a color $i < \ell$ such that for all $x \in (p)^k$, $x \in O_i$.
\item $\cDRT^k_\ell$: For every coloring $c:(<\omega)^{k-1}\rightarrow \ell$, there is a partition $p \in (\omega)^\omega$
  and a color $i < \ell$ such that for all $x \in (p)^k$, $c(x^\ast) = i$.
\end{enumerate}
\end{defn}

Our first goal is to show that the instances of $\cDRT^k_\ell$ are in one-to-one canonical correspondence with those instances of
$\ODRT^k_l$ for which the coloring of $(\omega)^k$ is \emph{reduced}. We define a reduced coloring without considering the coding method 
and note that any reduced coloring is open. 

\begin{defn}
\label{reduced-coloring} 
Let $y \in (\omega)^\omega$ and $m< k$.  A coloring of $(y)^k$ is
\emph{$m$-reduced} if whenever $p,q \in (y)^k$ and $p\restrict
\mu^p(m) = q\restrict \mu^q(m)$, $p$ and $q$ have the same color.
A coloring of $(y)^k$ is \emph{reduced} if it is $(k-1)$-reduced.
\end{defn}  

Note that a coloring is reduced means that the color of 
each partition $p \in (y)^k$ depends only on $p^\ast$.

\begin{prop}[$\RCA$]
\label{prop:2.2}
The following are equivalent.
\begin{enumerate}
\item $\CDRT^k_\ell$.
\item For every open reduced coloring $(\omega)^k=\cup_{i<\ell} O_i$,
there are $p\in(\omega)^\omega$ and $i<\ell$ such that $(p)^k \subseteq O_i$.
\item For every $y \in (\omega)^{\omega}$ and open reduced coloring $(y)^k = \cup_{i<\ell} O_i$, there
  are $p\in(y)^\omega$ and $i<\ell$ such that  $(p)^k\subseteq O_i$.
\end{enumerate}
\end{prop}

\begin{proof}  Clearly (3) implies (2).  To see that (2) implies (3), fix $y \in (\omega)^{\omega}$ and 
a reduced open coloring $(y)^k = \cup_{i<\ell} O_i$. 
Define 
$$\widehat{O}_i = \{ \langle n, \tau\rangle : \tau \in (<\omega)^k \text{ and } 
\tau \circ y \in O_i\}$$
 It is straightforward to check that the coloring $(\omega)^k = \cup_{i < \ell} \widehat{O}_i$ 
is also reduced, and that whenever $x$ is homogeneous for $\cup_{i<\ell} \widehat{O}_i$
then $x\circ y$ is homogeneous for $\cup_{i<\ell} O_i$.

To see (2) implies (1), fix $c:(<\omega)^{k-1}\rightarrow \ell$. For each $i < \ell$, let
\[
O_i = \{\langle 0, \sigma^{\smallfrown}(k-1)\rangle : \sigma \in(<\omega)^{k-1} \text{ and }  c(\sigma) = i \}.
\]
Then $(\omega)^k = \cup_{i < \ell} O_i$ is an open 
reduced coloring of $(\omega)^k$, and any infinite partition 
which is homogeneous for it is also homogeneous for $c$.

For the implication from (1) to (2), assume $\CDRT^k_\ell$, 
and suppose we are given a coloring $\cup_{i<\ell} O_i$.
Now, for each $\sigma\in (<\omega)^{k-1}$, we 
define $c(\sigma)$ as follows.  Note that for some $i<\ell$, some 
$\tau\succeq \sigma^\smallfrown(k-1)$, and some $n$, 
we have $\langle n, \tau \rangle \in O_i$.  Letting $\langle n, \tau, i\rangle$
be the least triple with this property, we define $c(\sigma)=i$.

Let $i < \ell$ and $p \in (\omega)^{\omega}$ be the result of 
applying $\CDRT^k_\ell$ to $c$.  Given $x \in (p)^k$, we 
know that $c(x^\ast) = i$.  Let $n, \tau$ be the witnesses used 
in the definition of $c(x^\ast)$.  Let $q \in (\omega)^k$ with 
$q \succ \tau$.  Then $q \in O_i$.  Since $O_i$ 
is reduced and $q^\ast = \tau^\ast = x^\ast$, $x \in O_i$.
Therefore, $p$ is homogeneous for the coloring $\cup_{i<\ell} O_j$, 
as required.
\end{proof}

It is now routine to show that the number of colors does not matter.

\begin{prop}[$\RCA$]
$\cDRT^k_{\ell}$ and $\cDRT^k_{2}$ are equivalent.
\end{prop}
\begin{proof}
Collapse colors and iterate $\cDRT^k_2$ finitely many times, 
using Proposition \ref{prop:2.2}.
\end{proof}

The next proof is essentially an effective version of an argument in \cite{pv}.

\begin{thm}[$\RCA$]
\label{thm:2.4}
$\BaDRT^k_{\ell}$, $\ODRT^k_\ell$ and $\cDRT^k_\ell$ are equivalent.
\end{thm}

\begin{proof}
By setting $D_n = (\omega)^k$ in $\BaDRT^k_{\ell}$, $\ODRT^k_\ell$ is a special case of $\BaDRT^k_\ell$, and by Proposition \ref{prop:2.2}, $\cDRT^k_\ell$ 
is a special case of $\ODRT^k_\ell$.  It remains to prove in $\RCA$ that $\cDRT^k_\ell$ implies $\BaDRT^k_\ell$. 

Let $\{O_i\}_{i<\ell}$, $\{D_n\}_{n<\omega}$ be a Baire $\ell$-coloring of $(\omega)^k$ for which the open sets $O_i$ are pairwise disjoint. We construct 
a partition $y \in (\omega)^\omega$ such that $(y)^k
\subseteq \cap_n D_n$ and $\cup_i O_i$ restricted to $(y)^k$ is
reduced. By Proposition \ref{prop:2.2} and $\cDRT^k_\ell$, there is a homogeneous 
$z \in (y)^{\omega}$ for this open reduced coloring. Since $(z)^k \subseteq (y)^k \subseteq \cap_n D_n$, this  
partition $z$ is homogeneous for the original Baire coloring. 

First we describe the construction in a classical way, and then remark
on how it can be carried out in $\RCA$.  

Build $y$ by 
initial segments in stages, $y = \lim_s y_s$, starting with $y_{0}$ being the empty string,
and then continuing with stage $s=1$ as follows.  
Assume that at the start of stage $s$, 
$y_{s-1}$ is an $(s-1)$-partition.  In stage $s$ begin by letting $y_s^0 = {y_{s-1}}\concat (s-1)$, so that $y_s^0$ is an $s$-partition.
Let $x_0,\dots, x_r$ be a list of the elements of $(s)^k$.  
For each $i=0,\dots r$,
let $q = x_i \circ y_s^i$.  
Let $\tau \in (<\omega)^k$ be such that $q \preceq \tau$ and $\tau$
meets $\cap_{n\leq s} D_n$ and $\cup_{i<\ell} O_i$. 
Then extend $y_s^i$ to $y_s^{i+1}$ in such a way that 
$x_i \circ y_s^{i+1} =  \tau$.  In general there is more than one way to do this, 
but which way does not matter.  For concreteness, for each $n\geq|y_s^i|$ we could 
set $y_s^{i+1}(n)$ to be the least $m$ such that $x_i(m) = \tau(n)$.
At the conclusion of these substages we are 
left with $y_s^{r+1}$.   Let $y_s = y_s^{r+1}$.  This completes the construction of $y$.

We need to justify why this construction can be carried out in $\RCA$.
To that end, we make the following claims in $\RCA$:
\begin{enumerate}
\item For any $q \in (<\omega)^k$ and $s$, there is 
an extension $\tau \succeq q$ which meets $\cup_{i<\ell} O_i$ 
and $\cap_{n\leq s} D_n$.
To see that for all $s$, such a $\tau$ exists, apply $\Sigma^0_1$ induction.
\item There is a function $f:(<\omega)^k\times \omega\rightarrow (<\omega)^k$ 
with the properties above.
This follows because in $\RCA$, we can select the $\tau$ with least witness.
\item There is a function which outputs the sequence 
$$y_1^0,\dots,y_1^{r_1},y_2^0,\dots,y_2^{r_2},y_3^0,\dots$$
This can be obtained by primitive recursion using the function $f$.
\end{enumerate}

Therefore, $y$ exists in $\RCA$.  Next we show 
that $(y)^k\subseteq \cap_n D_n$.
Let $w \in (y)^k$ and fix $n$.  Let $x \in (\omega)^k$ with $x\circ y = w$.
Let $s\geq n$ be large enough that 
$x\restrict s \in (s)^k$.  Then $x \restrict s$ was one of the ways to coarsen
considered during stage $s$ of the construction.  By construction,
$(x \restrict s) \circ y_{s}$ 
meets $D_n$.  So $x\circ y\in D_n$.

Finally, we claim that the restriction of $\cup_{i<\ell} O_i$
to $(y)^k$ is a reduced coloring.  Given
if $w_1,w_2 \in (y)^k$ with $w_1^\ast=w_2^\ast$, 
let $x_1$ and $x_2$
be such that $x_1\circ y = w_1$ and $x_2\circ y = w_2$.  
Then $x_1^\ast = x_2^\ast$.
Let $x=(x_1^\ast)\concat(k-1)$
and let $s = |x|$.  Then $x \in (s)^k$ and $x$ was considered 
at stage $s$ of the construction.  By construction,
$x \circ y_s$ meets $O_i$ for some $i$.  Since $x \circ y_s$
is an initial segment of both $w_1$ and $w_2$, it follows that 
$w_1$ and $w_2$ are both in $O_i$.  Finally, as $w_1,w_2 \in \cap_n D_n$,
we have $w_1,w_2 \in C_i$, as needed.
\end{proof}

Since $\ODRT^{k+1}_\ell$ implies $\RT^k_\ell$ over $\RCA$ \cite{ms}, we have the following corollary.

\begin{cor}[$\RCA$]
\label{cor:2.6}
$\cDRT^{k+1}_\ell$ implies $\RT^k_\ell$.
\end{cor}

\begin{prop}
\label{prop:2ll_computable}
For any $\ell \geq 2$, $\RCA$ proves $\cDRT^2_\ell$ and hence also $\ODRT^2_{\ell}$. 
\end{prop}
\begin{proof} 
Let $c: (<\omega)^1 \rightarrow \ell$. Since $(<\omega)^1 = \{ 0^n : n \in \omega \}$, 
$c$ can be viewed as an $\ell$-coloring of $\omega$.  By
$\RT^1_\ell$, there is a color $i$ and an infinite set $X$
such that for every $n \in X$, $c(0^n) = i$.  Let $z$ be the partition
which has a block of the form $\{n\}$ for each $n \in X$ and puts all
the other numbers in $\block{z}{0}$. Then $z$ is homogeneous for $c$. 
\end{proof}

\subsection{Connections to Hindman's theorem}
\label{subsec:Hindman}

In this section, we show that Hindman's Theorem for $\ell$-colorings implies
$\CDRT^3_\ell$.  In \cite{cs}, Simpson remarks that 
one case of the Carlson-Simpson Lemma follows from 
Hindman's Theorem.  Ludovic Patey showed us a proof, and the 
same argument gives a strong form of $\CDRT^3_\ell$.  We 
include Patey's proof here.

\begin{defn}[$\RCA$]
Let $\mathcal{P}_{\text{fin}}(\omega)$ denote the set of (codes for) all non-empty
finite subsets of $\omega$. $X \subseteq
\mathcal{P}_{\text{fin}}(\omega)$ is an \emph{IP set} if $X$ is closed
under finite unions and contains an infinite sequence of pairwise
disjoint sets. 
\end{defn}

\begin{thm}[Hindman's theorem for $\ell$-colorings]
	For every $c : \mathcal{P}_{\text{fin}}(\omega) \to \ell$
        there is an $IP$ set $X$ and a color $i < \ell$ such that
        $c(F) = i$ for all $F \in X$.
\end{thm}

\begin{thm}[essentially Patey \cite{patey}, see also \mbox{\cite[page 268]{cs}}]
\label{prop:hindman} Over $\RCA$,
	Hindman's theorem for $\ell$-colorings implies
        $\cDRT^3_\ell$. In particular, $\cDRT^3_\ell$ is provable in
        $\ACA^+$.
\end{thm}

\begin{proof} Hindman's Theorem follows from $\ACA^+$ by \cite{bhs},
so it suffices to prove the first statement.
Fix $\ell \geq 2$ and assume Hindman's Theorem for $\ell$-colorings. Since Hindman's Theorem for $2$-colorings implies $\ACA$, we reason in $\ACA$. 
By Proposition \ref{prop:2.2}, it suffices to fix an open reduced coloring $(\omega)^3 = \cup_{i < \ell} O_i$ and produce 
$p \in (\omega)^\omega$ and $i < \ell$ such that for all $x \in (p)^3$, $x \in O_i$. We write the coloring as 
$c : (\omega)^3 \to \ell$ with the understanding that $c(x)=i$ is shorthand for $x \in O_i$. 

For a nonempty finite set $F \subseteq \omega$ with $0 \notin F$ and a number $n > \max F$, we let $x_{F,n} \in (\omega)^3$ be the following partition.
	\[
		x_{F,n}(k) =
		\begin{cases}
			0 & \text{if } k \notin F \text{ and } k \neq n\\
			1 & \text{if } k \in F\\
			2 & \text{if } k = n
		\end{cases}
	\]
Thus, the blocks are $\omega - (F \cup \{ n \})$, $F$ and $\{ n \}$. Note that we can determine the color $c(x_{F,n})$ as a function of $F$ and 
$n$ and that since $c$ is reduced, if $x \in (\omega)^3$ and $x \restrict \mu^x(2) = x_{F,n} \restrict n$, then 
$c(x) = c(x_{F,n})$.  
	
The remainder of the proof is most naturally presented as a forcing construction. After giving a classical description of this construction, we indicate how 
to carry out the construction in $\ACA$. The forcing conditions are pairs
$(F,I)$ such that
	\begin{itemize}
		\item $F$ is a non-empty finite set such that $0 \notin F$,
		\item $I$ is an infinite set such that $\max F < \min I$, and
		\item for every nonempty subset $U$ of $F$ there is an
                  $i < \ell$ such that $c(x_{U,n}) = i$ for all $n \in
                  F \cup I$ with $\max U < n$.
	\end{itemize}
Extension of conditions is defined as for Mathias forcing:
$(\widehat{F},\widehat{I}) \leq (F,I)$ if $F \subseteq \widehat{F}
\subseteq F \cup I$ and $\widehat{I} \subseteq I$.
	
By the pigeonhole principle, there is an $i < l$ such that
$c(x_{\{1\},n}) = i$ for infinitely many $n > 1$. For any such
$i$, the pair $(\{1\}, \{n \in \omega : n > 1 \text{ and }
c(x_{\{1\},n}) = i\})$ is a condition. More generally, given a
condition $(F,I)$ there is an infinite set $\widehat{I} \subseteq I$
such that $(F \cup \{\min I\}, \widehat{I})$ is also a condition. To
see this, let $U_0,\ldots,U_{s-1}$ be the nonempty subsets of $F
\cup \{\min I\}$ containing $\min I$. By arithmetic induction, for each positive $k \leq s$, there exist colors
$i_0,\ldots,i_{k-1} < \ell$ such that there are infinitely many $n \in
I$ with $c(x_{U_j,n}) = i_j$ for all $j < k$. (If not, fix the least
$k$ for which the fact fails, and apply the pigeonhole principle to
obtain a contradiction.) Let $i_0,\ldots,i_{s-1}$ be the colors
corresponding to $k = s$ and let $\widehat{I}$
be the infinite set $\{n \in I : \forall j < s~(c(x_{U_j,n}) =
i_j)\}$.

Fix a sequence of conditions $(F_1,I_1) > (F_2,I_2) > \cdots$ 
with $|F_k| = k$ and let $G = \bigcup_k F_k$. To complete the proof, we use 
$G$ to define a coloring $d: \mathcal{P}_{\fin}(\omega) \rightarrow \ell$ to which we can 
apply Hindman's Theorem. However, first we indicate why we can form $G$ in $\ACA$. 

The conditions $(F,I)$ used to form $G$ can be specified by the finite set $F$, the number $m = \min I$ and the finite sequence 
$\delta \in \ell^{M}$ where $M = 2^{|F|}-1$ such that if $F_0, \ldots, F_{M-1}$ is a canonical listing of the nonempty subsets of $F$, then 
$I = \{ n \geq m : \forall j < M \,  (c(x_{F_j,n}) = \delta(j)) \}$. The extension procedure above can be captured by an 
arithmetically definable function $f(F,m,\delta) = \langle F \cup \{m \}, m', \delta' \rangle$ where $F \cup \{ m \}$, $m'$ and $\delta'$ describe 
the extension $(F \cup \{ m \}, \widehat{I})$. Because the properties of this extension where verified using arithmetic 
induction and the pigeonhole principle, both of which are available in $\ACA$, we can define $f$ in $\ACA$ and form 
a sequence of conditions $(F_1, m_1, \delta_1) > (F_2, m_2, \delta_2) > \cdots$ giving $G = \bigcup_k F_k$. 

It remains to use $G = \{ g_0 < g_1 < \cdots \}$ to complete the proof. By construction, for each non-empty finite subset $U$ of $G$, there
is color $i_U < \ell$ such that $c(x_{U,n}) = i_U$ for all $n \in G$
with $n > \max U$. Define $d :
\mathcal{P}_{\text{fin}}(\omega) \to \ell$ by $d(F) = i_{\{g_m : m \in
  F\}}$. We apply Hindman's theorem to $d$ to
obtain an IP set $X$ and a color $i < \ell$. Since $X$ contains an
infinite sequence of pairwise disjoint members, we can find a sequence
$E_1,E_2,\ldots$ of members of $X$ such that $\max E_k < \min E_{k+1}$. Define $p \in (\omega)^\omega$ to be the partition whose
blocks are $\block{p}{0} = \omega - \bigcup_k \{g_m : m \in E_k\}$ and, for each $k \geq 1$,
$\block{p}{k} = \{g_m : m \in E_k\}$. Note that for all $k \geq 1$, 
\[
\max \block{p}{k} = \max \{g_m : m \in E_k\} < \min \block{p}{k+1} = \min \{g_m : m \in E_{k+1}\}.
\]
	
It remains to verify that $p$ and $i$ have the desired
properties. Consider any $x \in (p)^3$; we must show that $c(x) =
i$. Let $U = \block{x}{1} \cap \mu^x(2)$ and let $n = \mu^x(2)$. 
Then $n = \mu^{x_{U,n}}(2)$ and $x \restrict n = x_{U,n}
\restrict n$, so since $c$ is reduced, $c(x) = c(x_{U,n})$. Therefore, it
suffices to show $c(x_{U,n}) = i$.

We claim $U$ is a finite union of $p$-blocks. Because $x$ is a coarsening of $p$, 
$\block{x}{1}$ is a (possibly infinite) union of $p$-blocks $\block{p}{j_1} \cup \block{p}{j_2} \cup \cdots$ 
with $0 < j_1 < j_2 < \cdots$ and $n = \mu^x(2) = \min \block{x}{2} = \min \block{p}{b}$ for some $b \geq 2$. 
Let $j_a$ be the largest index such that $j_a < b$. 
Since the $p$-blocks are finite and increasing, $U = \block{x}{1} \cap \mu^x(2) = \block{p}{j_1} \cup \cdots \cup \block{p}{j_a}$. 
Note that $n \in G$ (because $\block{p}{b} \neq \block{p}{0}$) and $\max U < n$. 

It follows that $U = \{ g_m \mid m \in F \}$ where $F = E_{j_1} \cup \cdots \cup E_{j_a}$. Since our fixed IP set $X$ is 
closed under finite unions, $F \in X$ and therefore $d(F)=i$. By the definition of $d$, 
$d(F) = i_{\{g_m \mid m \in F\}} = i_U$, so $i = i_U$. Finally, $U$ is a finite subset of $G$, $n \in G$ and 
$\max U < n$, so $c(x_{U,n}) = i_U = i$ as required. 
\end{proof}

Observe that this proof of $\CDRT^3_\ell$ from $\HT$ produces a homogeneous $p$ with a special property: 
$\max \block{p}{i} < \min \block{p}{i+1}$ for all $i>0$.
We show that this strengthened ``ordered finite block'' version of $\CDRT^3_\ell$ is equivalent to $\HT$.
However, there is no finite block version of $\CDRT^k_\ell$ for $k>3$.

\begin{prop}[$\RCA$]
\label{prop:2.14}
If for every $\ell$-coloring of $(<\omega)^2$ there is an infinite homogeneous partition $p$ with
$\max \block{p}{i} < \min \block{p}{i+1}$ for all $i>0$, then Hindman's Theorem for $\ell$-colorings holds.
\end{prop}

\begin{proof}
Given $c: \mathcal{P}_{\text{fin}}(\omega) \rightarrow \ell$, define $\widehat{c}: (<\omega)^2 \rightarrow \ell$ by 
$\widehat{c}(\sigma) = c(\{i<|\sigma| : \sigma(i) = 1\})$.  Let $p$ be a homogeneous partition for $\widehat{c}$ with $\max \block{p}{i} < \min \block{p}{i+1}$ for all $i>0$.  
The set of all finite unions of the blocks $\block{p}{i}$ for $i>0$ satisfies the conclusion of Hindman's Theorem.
\end{proof}

\begin{prop}
\label{prop:2.13}
There is a 2-coloring of $(<\omega)^3$ such that any infinite homogeneous partition $p$ has $\block{p}{i}$ infinite for all $i > 0$.
\end{prop}

\begin{proof}
For $\sigma \in (<\omega)^3$, set $c(\sigma) = 1$ if $\sigma$ contains more 1's than 2's and set $c(\sigma)= 0$ 
otherwise. Let $p$ be homogeneous for this coloring.  Suppose for 
contradiction that $i > 0$ is such that $\block{p}{i}$ is finite. Let $N = i + 2 + |\block{p}{i}|$ and let $x = w\circ p$ where 
$$w(n) = \begin{cases} 1 &\text{ if } n=i\\
2 & \text{ if } i < n \leq N\\
3 & \text{ if } n = N+1\\
0 & \text{ otherwise}\end{cases}$$
Since $x^\ast$ has more 2's than 1's, 
$c(x^\ast) = 0$. Now coarsen in a different way: let $h \in [i+1,N]$ 
be chosen so that the size of $\block{p}{h}\cap [0,\mu^x(3)]$ is minimized. 
Let $y= z \circ p$ where
$$z(n) = \begin{cases} 1 & \text{ if } i\leq n \leq N \text{ and } n \neq h\\
                                    2 & \text{ if } n=h\\
                                    3 & \text{ if } n = N+1\\
                                    0 & \text{ otherwise.}\end{cases}$$
Since at least one $p$-block has moved from $\block{x}{2}$ to $\block{y}{1}$ and since $\block{y}{2}$ contains only the smallest $p$-block from $\block{x}{2}$, $c(y^\ast) = 1$.  So $p$ was not homogeneous.
\end{proof}

\subsection{$\cDRT$ and the Carlson-Simpson Lemma}\label{sec:2.3}

The Carlson-Simpson Lemma is the main technical tool in the original proof of the Borel version of the 
Dual Ramsey Theorem. The principle is usually stated in the framework of variable
words, but it can also be understood as a special case of the Combinatorial Dual Ramsey Theorem.  

\begin{carlsonsimpsonlemma}[$\CSL(m,\ell)$]
For every coloring $c:(<\omega)^{m} \rightarrow \ell$, there is a
partition $p\in(\omega)^\omega$ and a color $i$ such that for all $x
\in (p)^{m+1}$, if $\block{p}{j} \subseteq \block{x}{j}$ for each $j<m$, then 
$c(x^\ast) = i$.
\end{carlsonsimpsonlemma}

The condition $\block{p}{j} \subseteq \block{x}{j}$ for $j < m$ captures those $x \in (p)^{m+1}$ which keep the first $m$ many blocks of $p$ distinct in $x$. Therefore, 
$\CSL(m,\ell)$ is a special case of $\CDRT^{m+1}_{\ell}$.
Two related principles, 
$\OVW(m,\ell)$ and $\VW(m,\ell)$ have also been studied 
(see \cite{ms, e, LiuMoninPatey}).  We do not deal with 
these principles, but it may be useful to note that $\VW(m,\ell)$ 
is the strengthening of $\CSL(m,\ell)$ which requires 
each nonzero block $\block{p}{j}$ to be finite, and $\OVW(m,\ell)$ 
is the further strengthening which requires 
$\max \block{p}{j} < \min \block{p}{j+1}$ for all $j>0$.

In Proposition \ref{prop:csl_versions}, we give three equivalent versions of the Carlson-Simpson Lemma. 
The version in Proposition \ref{prop:csl_versions}(2) is (up to minor notational changes which are easily translated in $\RCA$) the statement from Lemma 2.4 of
Carlson and Simpson \cite{cs}. 

\begin{prop}[$\RCA$]
\label{prop:csl_versions}
The following are equivalent. 
\begin{enumerate}
\item[(1)] $\CSL(m,\ell)$.
\item[(2)] For each coloring $c:(<\omega)^m \rightarrow \ell$, there is a partition $p\in(\omega)^\omega$ and a color $i$ such that for all $j<m$, $p(j)=j$ and 
for all $x \in (p)^{m+1}$, if $\block{p}{j} \subseteq \block{x}{j}$ for each $j<m$, then $c(x^\ast)=i$.
\item[(3)] For each $y\in(\omega)^\omega$ and open reduced coloring $(y)^{m+1} = \cup_{i<\ell} O_i$, there is a partition $p \in (y)^\omega$ and a color $i$ such that 
for all $j < m$, $\block{y}{j}\subseteq \block{p}{j}$ and for all $x \in (p)^{m+1}$, if $\block{p}{j} \subseteq \block{x}{j}$ for each $j<m$, then $x \in O_i$.
\end{enumerate}
\end{prop}

\begin{proof}
(2) implies (1) because $\CSL(m,\ell)$ is a special case of (2). The extra condition in (2) that $p(j)=j$ for $j < m$ says that the partition $p$ does not collapse any 
of the first $m$-many blocks of the trivial partition defined by the identity function. The equivalence between (2) and (3) is proved in a similar way to Proposition \ref{prop:2.2}. 

It remains to prove (1) implies (2). Fix an $\ell$-coloring $c: (<\omega)^m \rightarrow \ell$. 
Define $\tilde{c}:(<\omega)^m \rightarrow \ell$ by $\tilde{c}(\sigma) = c(0^{\smallfrown}1^{\smallfrown}\cdots^{\smallfrown}(m-1)^{\smallfrown}\sigma)$. 
Apply $\CSL(m,\ell)$ to $\tilde{c}$ to get $\tilde{p} \in (\omega)^{\omega}$ and $i < \ell$ such that for all $\tilde{x} \in (\tilde{p})^{m+1}$, if 
$\block{\tilde{p}}{j} \subseteq \block{\tilde{x}}{j}$ for all $j<m$, then $\tilde{c}(\tilde{x}^\ast) = i$. 

Let $p \in (\omega)^{\omega}$ be the partition defined by
$$p(j) = \begin{cases} j & \text{ if } j< m \\ \tilde{p}(j-m) &\text{ if } j\geq m\end{cases}.$$
We claim that $p$ satisfies the conditions in (2) for the coloring $c$ with the fixed color $i$.
Fix $x \in (p)^{m+1}$ such that $\block{p}{j} \subseteq \block{x}{j}$ for all $j < m$. We need to show that $c(x^\ast) = i$. 
Since $x$ does not collapse any of the first $m$-many $p$-blocks, 
$x(j)=j$ for all $j < m$.  Define $\tilde x \in (\tilde p)^{m+1}$ by $\tilde x(j) = x(j+m)$.
Then $\block{\tilde{p}}{j} \subseteq 
\block{\tilde{x}}{j}$ for all $j < m$. Therefore, $\tilde{c}(\tilde{x}^\ast) = i$. 
Now, $x^\ast = 0\concat 1 \concat \dots \concat (m-1)\concat \tilde{x}^\ast$.
Therefore, $c(x^\ast) = \tilde{c}(\tilde{x}^\ast) = i$,
as required to complete the proof that (1) implies (2). 
\end{proof}

Let $y \in (\omega)^{\omega}$ and $(y)^k = \cup_{i < \ell} C_i$ be an $m$-reduced coloring for some $1 < m < k$. We define the \textit{induced coloring} 
$(y)^{m+1} = \cup_{i < \ell} \widehat{C}_i$ as follows. For $\widehat{q} \in (y)^{m+1}$, $\widehat{q} \in \widehat{C}_i$ if and only if $q \in C_i$ for some 
(or equivalently all) $q \in (y)^k$ such that $\widehat{q} \upharpoonright \mu^{\widehat{q}}(m) = q \upharpoonright \mu^q(m)$. 
This induced coloring is a reduced coloring of $(y)^{m+1}$ and therefore we 
can apply $\CSL(m,\ell)$ to it. 

Our proof of $\CDRT^k_\ell$ from the Carlson-Simpson Lemma will use repeated applications of the following lemma, which is proved using $\omega$ many nested 
applications of $\CSL(m,\ell)$. 

\begin{lem}\label{lem:2.10}
Fix $1<m<k$ and $y \in (\omega)^\omega$.  Let $(y)^k = \cup_{i < \ell} C_i$ be an 
$m$-reduced coloring.  There is an $x \in (y)^\omega$ such
that the coloring restricted to $(x)^k$ is $(m-1)$-reduced.
\end{lem}

\begin{proof}
Fix an $m$-reduced coloring $(y)^{k} = \cup_{i < \ell} C_i$. We define a sequence of infinite partitions $x_m, x_{m+1}, \cdots$ starting with index $m$ 
such that $x_m = y$ and $x_{s+1}$ is a coarsening of $x_s$ for which $\block{x_s}{j} \subseteq \block{x_{s+1}}{j}$ for all $j < s$. 
That is, we do not collapse any of the first $s$-many blocks of the partition $x_s$ when we coarsen it to $x_{s+1}$. This property guarantees that 
the sequence has a well-defined limit $x \in (\omega)^{\omega}$. We show this limiting partition $x$ satisfies the conclusion of the lemma. 

Assume $x_s$ has been defined for a fixed $s \geq m$ and we construct $x_{s+1}$. 
Set $x_s^0 = x_s$. Let $\sigma_0,\dots,\sigma_r$ be a list of the elements of 
$(s)^m$.  We define a sequence of coarsenings $x_s^1, \ldots, x_s^r$ and set $x_{s+1} = x_s^r$. 

Assume that $x_s^j$ has been defined.  Define
$$\sigma_j^+(n) = \begin{cases} \sigma_j(n) & \text{ if } n < s \\ m + (n-s) & \text{ if } n \geq s\end{cases},$$
and let $w_s^j = \sigma_j^+ \circ x_s^j$.  That is $w_s^j$ collapses 
the first $s$-many blocks of $x_s^j$ into 
$m$-many blocks in the $j$-th possible way and leaves the remaining blocks of 
$x_s^j$ unchanged. Since $w_s^j$ is a coarsening of $y$, the coloring 
$\cup_{i < \ell} C_i$ is also an $m$-reduced coloring of $(w_s^j)^k$.
Let $(w_s^j)^{m+1} = \cup_{i<\ell} \widehat C_i$
be the induced coloring.  This coloring is reduced, so
let $i_s^j<\ell$ and $z_s^j \in (w_s^j)^\omega$ 
be the result of applying $\CSL(m,l)$ as stated in 
Proposition \ref{prop:csl_versions}(3).  Then $z_s^j$ 
leaves the first $m$ blocks of $w_s^j$ separate, 
and any coarsening of $z_s^j$ into at least $m+1$ pieces
receives color $i_s^j$, provided the first $m$ blocks 
are left separate.

To define $x_s^{j+1}$, we want to ``uncollapse'' the first $m$-many blocks of $z_s^j$ to reverse the action of $\sigma_{j}^+$ in defining $w_s^j$. Since 
$w_s^j$ collapsed the first $s$-blocks of $x_s^j$ to $m$-many blocks and since $z_s^j$ is a coarsening of $w_s^j$, if $x_s^j(u) < s$, then $z_s^j(u) < m$. 
We define $x_s^{j+1}$ by cases as follows.
\begin{enumerate}
\item[(1)] If $x_s^j(u) < s$, then $x_s^{j+1}(u) = x_s^j(u)$. 
\item[(2)] If $x_s^j(u) \geq s$ and $z_s^j(u) = a < m$, then $x_s^{j+1}(u) = x_s^j(\mu^{z_s^j}(a))$.
\item[(3)] If $z_s^j(u) \geq m$, then $x_s^{j+1}(u) = z_s^j(u)+(s-m)$.
\end{enumerate}
Below we verify that $x_s^{j+1}$ is an infinite partition coarsening $x_s^j$ which does not collapse any of the first $s$-many blocks of $x_s^j$. 
This completes the construction of $x_s^{j+1}$ and hence of $x_{s+1}$ and $x$.

We verify the required properties of $x_s^{j+1}$. By (1), $\block{x_s^j}{a} \subseteq \block{x_s^{j+1}}{a}$ for all $a < s$, so we do not collapse any of the first $s$-many blocks 
of $x_s^j$ in $x_s^{j+1}$. There is no conflict between (1) and (3) because $x_s^j(u) < s$ implies $z_s^j(u) < m$. Furthermore, (3) renumbers the 
$z_s^j$-blocks starting with index $m$ to $x_s^{j+1}$-blocks starting with index $s$ without changing any of these blocks. Therefore, $x_s^{j+1}$ is an infinite partition. 

In (2), we handle the case when the $x_s^j$-block containing $u$ is not changed by $w_s^j$ (except to renumber its index) but is collapsed by 
$z_s^j$ into one of the first $m$-many $z_s^j$-blocks. In this case, $\mu^{z_s^j}(a) = \mu^{x_s^j}(b)$ for some $b < s$ and we have set $x_s^{j+1}(u) = b$. 
It is straightforward to check (as in the proof of Theorem \ref{thm:2.4}) that $x_s^{j+1}$ is a coarsening of $x_s^j$ and that $\sigma_{j}^+ \circ x_s^{j+1} = z_s^j$.

To complete the proof, we verify that the restriction of $\cup_{i < \ell} C_i$ to $(x)^k$ is
$(m-1)$-reduced.  Fix $p \in (x)^k$ and we show the color of $p$ depends only on $p\restrict \mu^p(m-1)$.

Let $q \in (\omega)^\omega$ be the unique element with
$p = q \circ x$, and let $\sigma = q \restrict \mu^q(m-1)$.
Then $ \sigma\concat (m-1) \in (s)^m$ for some $s$, and
$$p\restrict \mu^p(m-1) = \sigma \circ (x\restrict\mu^x(s-1)).$$
During stage $s$ and afterward, the first $s$ blocks of $x_s$ are 
always kept separate.  Therefore, the above equation remains true when $x$
is replaced with $x_s^{j+1}$, where $j$ is the unique index such 
that $\sigma_j = \sigma$.  Therefore, $p$ is a coarsening of 
$\sigma_j^+ \circ x_j^{s+1} = z_s^j$ and $p$ keeps the first $m$ 
blocks of $z_s^j$ separate.  
Therefore, the color of $p$ is $i_s^j$, the homogeneous color 
obtained when we applied $\CSL(m,\ell)$ to obtain $z_s^j$. 
This completes the proof that 
the restriction of $\cup_{i < \ell} C_i$ to $(x)^k$ is $(m-1)$-reduced because the indices $s$ and $j$ in $z_s^j$ are determined only by 
$p\restrict \mu^p(m-1)$.
\end{proof}

We end this section with the proof of $\CDRT^k_{\ell}$.
\begin{thm}\label{thm.3.31} For all for $k \geq 2$ and all $\ell$, $\CDRT^k_\ell$ holds.
\end{thm}
\begin{proof} For $k=2$, $\CDRT^k_{\ell}$ follows from the pigeonhole 
principle as in Proposition \ref{prop:2ll_computable}. 
Now assume $k\geq 3$.  
Consider $\CDRT^k_{\ell}$ in the form given in Proposition \ref{prop:2.2}. 
Let $y \in (\omega)^{\omega}$ 
and $(y)^k = \cup_{i < \ell} O_i$ be an open reduced coloring. These satisfy the assumptions of Lemma \ref{lem:2.10} with $m=k-1$.  After $k-2$ applications of 
Lemma \ref{lem:2.10}, we obtain $x \in (y)^{\omega}$ such that the restriction of $\cup_{i < \ell} O_i$ to $(x)^k$ is $1$-reduced and hence the color of $p \in (x)^k$ 
depends only on $p \upharpoonright \mu^p(1)$. Since the numbers $n < \mu^p(1)$ must lie in $\block{p}{0}$, the color of $p$ is determined by the value of 
$\mu^p(1)$. By the pigeonhole principle, there is an infinite set $X \subseteq \{ \mu^x(a) : a \geq 1 \}$ and a color $i$ such that for all $p \in (x)^k$, 
if $\mu^p(1) \in X$, then $p \in C_i$. It follows that for any $z \in (x)^{\omega}$ such that $\mu^z(a) \in X$ for all $a \geq 1$, 
$(z)^k \subseteq C_i$ as required. 
\end{proof}

It is interesting to note that the only non-constructive steps in this proof are the $\omega \cdot (k-2)$ nested applications of the 
Carlson-Simpson Lemma.  

\section{The Borel Dual Ramsey Theorem for $k \geq 3$}\label{sec:3}
\label{sec:Borel3}

In the next two sections we consider the Borel Dual Ramsey Theorem from the perspective of effective mathematics. We define Borel codes for topologically $\mathbf\Sigma^0_\alpha$ subsets of $(\omega)^k$ by induction on the ordinals below $\omega_1$. Let $L$ be some 
countable set of labels which effectively code for the clopen sets $\emptyset$, $(\omega)^k$ and $[\sigma]$ and $\overline{[\sigma]}$ for $\sigma \in (<\omega)^k$. 

\begin{defn} 
\label{defn:effective_borel_codes}
We define a \emph{Borel code for a $\mathbf \Sigma^0_\alpha$ or $\mathbf \Pi^0_\alpha$ set}.
\begin{itemize}
\item A \emph{Borel code for a $\mathbf \Sigma^0_0$ or a $\mathbf \Pi^0_0$ set} is a labeled tree $T$ consisting of just a root $\lambda$ in which the root is 
labeled by a clopen set from $L$. The Borel code represents that clopen set. 
\item For $\alpha \geq 1$, a \emph{Borel code for a $\mathbf\Sigma^0_{\alpha}$ set} is a labeled 
tree with a root labeled by $\cup$ and attached subtrees at level 1, each of which is 
a Borel code for a $\mathbf\Sigma^0_{\beta_n}$ or $\mathbf \Pi^0_{\beta_n}$ set $A_n$ for some 
$\beta_n<\alpha$.  The code represents the set $\cup_n A_n$.
\item For $\alpha \geq 1$, a \emph{Borel code for a $\mathbf\Pi^0_\alpha$ set} is 
the same, except the root is labeled $\cap$. The 
code represents the set $\cap_n A_n$.
\end{itemize}
For $\alpha \geq 1$, a \emph{Borel code for a $\mathbf\Delta^0_\alpha$ set} is a pair of labeled trees 
which encode the same set, where one encodes it as a $\mathbf\Sigma^0_\alpha$ set 
and the other encodes it as a $\mathbf\Pi^0_\alpha$ set.
\end{defn}

The codes are faithful to the Borel hierarchy in the sense that every code for a $\mathbf\Sigma^0_\alpha$ set represents a $\mathbf \Sigma^0_\alpha$ set and every $\mathbf\Sigma^0_{\alpha}$ 
set is represented by a Borel code for a $\mathbf\Sigma^0_{\alpha}$ set. 
There is a uniform procedure to transform a Borel code $B$ for a $\mathbf\Sigma^0_{\alpha}$ set  
$A$ into a Borel code $\overline{B}$ for a $\mathbf \Pi^0_{\alpha}$ set $\overline{A}$: leave the underlying 
tree structure the same, swap the $\cup$ and $\cap$ labels and replace the leaf labels by their complements. 

Observe also that a code for a $\mathbf \Sigma^0_1$ set essentially agrees 
with the definition of a code for an open set in Definition \ref{defn:reverseopen}
(up to a primitive recursive translation mapping elements of $\omega\times(<\omega)^k$
to leaves of a $\mathbf \Sigma^0_1$ code, and mapping each leaf of a 
$\mathbf \Sigma^0_1$ code to an element or sequence of elements 
of $\omega\times(<\omega)^k$).
The one difference is that we must include a leaf label of $\emptyset$
in the definition of a Borel code, so that the empty set 
has a $\mathbf\Sigma^0_1$ code.  Having included $\emptyset$ as a label,
we also include $(\omega)^k$ to keep complementation effective.

We recall some notation from hyperarithmetic theory. 
Let $\mathcal{O}$ denote Kleene's set of computable ordinal notations. 
The ordinal represented by $a \in \mathcal{O}$ is denoted $|a|_\mathcal{O}$, 
with $|1|_\KO = 0$, $|2^a|_\KO = |a|_\KO + 1$, and 
$|3\cdot 5^e|_\KO = \sup_j |\varphi_e(j)|_\KO$.
The $H$-sets are defined by effective transfinite
recursion on $\mathcal{O}$ as follows: $H_1 = \emptyset$, $H_{2^a} = H_a'$ and $H_{3 \cdot 5^a}
= \{ \langle i,j \rangle \mid i \in H_{\varphi_a(j)} \}$. 
The reader is referred to Sacks \cite{sacks} for more
details.  To use oracles that line up better than the $H$-sets do
with the levels of the Borel hierarchy, define
$$\emptyset_{(a)} = \begin{cases} H_a & \text{ if } |a|_\kO < \omega\\
H_{2^a} & \text{ otherwise.}\end{cases}$$
If $|a|_\kO = |b|_\kO = \alpha$, then $\emptyset_{(a)} \equiv_1 \emptyset_{(b)}$,
so we sometimes just write $\emptyset_{(\alpha)}$ in that situation.
As usual,
$\omega_1^{CK}$ denotes the least noncomputable ordinal.

Recall the standard effectivizations of the notions described above.
We say that a Borel code $B$ is \emph{computable} if it is computable
as a labeled subtree of $\omega^{<\omega}$.  We say $B$ is 
\emph{effectively} $\Sigma^0_\alpha$ (respectively effectively $\Pi^0_\alpha$)
if the root is labeled $\cup$ (respectively $\cap$) and additionally there 
is $a \in \kO$ with $|a|_\kO = \alpha$, and a computable labeling
of the nodes of $B$ with notations from $\{ b : b \leq_\kO a\}$,
such that the root is labeled with $a$ and each node has 
a label strictly greater than all its extensions.  

It is well-known that an open set of high hyperarithmetic complexity 
can be represented by a computable Borel code 
for a $\Sigma^0_\alpha$ set, where $\alpha$ is an appropriate
computable ordinal.  In the following proposition, we use a 
standard technique to make this correspondence explicit.  
Fix an effective 1-to-1 enumeration $\tau_n$ for the 
strings $\tau \in (<\omega)^k$. 

\begin{prop}
\label{prop:Borelcodes} 
There is a partial computable function $p(x,y)$ such that $p(a,e)$ is
defined for all $a \in \mathcal{O}$ and $e \in \omega$ and such
that if $a \in \mathcal{O}$ and 
$R = \bigcup \{ [\tau_n] : n \in W_e^{\emptyset_{(a)}} \}$, then
$\Phi_{p(a,e)}$ is a computable Borel code for $R$ as a
$\Sigma^0_{\alpha+1}$ set, where $\alpha=|a|_\kO$.
\end{prop}

\begin{proof}
We define $p(a,e)$ for all $e$ by effective transfinite recursion on
$a \in \mathcal{O}$. Let
$\Phi_{p(1,e)}$ be a Borel code for the open set $R = \bigcup \{
[\tau_n] : n \in W_e \}$.

For the successor step, consider 
$R = \bigcup \{ [\tau_n] : n \in W_e^{\emptyset_{(2^a)}} \}$.
Each set which is
$\Sigma^0_1$ in $\emptyset_{(2^a)}$ is $\Sigma^0_2$ in $\emptyset_{(a)}$ and for such sets, we can effectively pass from a $\Sigma^{0,\emptyset_{(2^a)}}_1$ index to a
$\Sigma^{0,\emptyset_{(a)}}_2$ description. Specifically, uniformly in $e$, we compute an index $e'$ such that for all oracles $X$,
$\Phi_{e'}^X(x,y)$ is a total $\{ 0,1 \}$-valued function and 
\[
n \in W_e^{X'} \text{ if and only if } \exists t \,
\forall s \geq t \, ( \Phi_{e'}^X(n, s) = 1).
\]
Let $R_t = \bigcup \{ [\tau_n] : \exists s \geq t \, (\Phi_{e'}^{\emptyset_{(a)}}(n,s)=0)\}$. 
$R_0 \supseteq R_1 \supseteq \cdots$ is a decreasing sequence of sets such that $x \not \in R$ if and only if 
$\forall t \, (x \in R_t)$. Therefore, $R = \cup_t \overline{R_t}$. Each set $R_t$ can be represented as 
$R_t = \bigcup \{ [\tau_n] : n \in W_{e_t}^{\emptyset_{(a)}} \}$, where $e_t$ is uniformly computable from $e$ and $t$.
Applying the induction hypothesis, we define $p(2^a,e)$ to encode a tree whose root is labeled by a union and whose $t$-th subtree at level 1 is the Borel code 
representing the complement of $\Phi_{p(a,e_t)}$.

For the limit step, consider $R = \bigcup \{ [\tau_n] : n \in W_e^{\emptyset_{(3 \cdot 5^d)}} \}$. Uniformly in $e$, we construct a sequence of indices
$e_t$ for $t \in \omega$ such that for all oracles $X$, $\Phi_{e_t}^X(x)$ converges if and only if $\Phi_e^X(x)$ converges and
only asks oracle questions about numbers in the first $t$ many columns of $X$. Let $R_t = \bigcup \{ [\tau_n] : n \in W_{e_t}^{\oplus_{i \leq t} \emptyset_{(\varphi_d(i))}} \}$ 
and note that $R = \cup_t R_t$.  We can effectively pass to a sequence of indices $e_t'$ such that
$R_t = \bigcup \{ [\tau_n] : n \in W_{e_t'}^{\emptyset_{(\varphi_d(t))}} \}$. By induction, each $p(\varphi_d(t),e_t')$ is the index for a computable Borel code for $R_t$ as a
$\Sigma^0_{2^{\varphi_d(t)}}$ set, so we may define $p(3 \cdot 5^d,e)$ to be the
index of a tree which has $\cup$ at the root and $\Phi_{p(\varphi_d(t),e_t')}$ as its subtrees.
Since $2^{\varphi_d(t)} <_\kO 3\cdot 5^d$ 
for all $t$, the resulting Borel code has the required height.
\end{proof}

To force the Dual Ramsey Theorem to output computationally powerful 
homogeneous sets, we use the following definition and 
a result of Jockusch \cite{Jockusch1968}.

\begin{defn}
For functions $f,g: \omega \rightarrow \omega$, we say $g$
\textit{dominates} $f$, and write $g \succeq f$, if $f(n) \leq g(n)$
for all but finitely many $n$.
\end{defn}

\begin{thm}[Jockusch \cite{Jockusch1968}, see also \mbox{\cite[Exercise 16-98]{Rogers_book}}]
\label{thm:fact}
For each computable ordinal $\alpha$, there is a function $f_\alpha$ such that $f_\alpha
\equiv_T \emptyset_{(\alpha)}$ and for every $g \succeq f_\alpha$, we have $\emptyset_{(\alpha)} \leq_T g$.
\end{thm}

In Theorem \ref{thm:hyp}, we use these functions $f_\alpha$ to show that for every
computable ordinal $\alpha$, there is a computable Borel code for a set
$R \subseteq (\omega)^3$ such that any homogeneous partition $p \in (\omega)^{\omega}$ for the coloring $(\omega)^3 = R \cup \overline{R}$ 
computes $\emptyset_{(\alpha)}$. 

\begin{thm}
\label{thm:main}\label{thm:self-moduli-theorem}
Let $A$ be a set and $f_A$ be a function such that $A \equiv_T f_A$
and for every $g \succeq f_A$, we have $A \leq_T g$. There is an
$A$-computable clopen coloring $(\omega)^3 = R \cup \overline{R}$ for which every homogeneous partition $p$ satisfies $p \geq_T A$.
\end{thm}

\begin{proof}
Fix $A$ and $f_A$ as in the statement of the theorem. Without loss of generality, we assume that if $n < m$, then $f_A(n) < f_A(m)$. 
For $x \in (\omega)^3$, let $a_x = \mu^x(1)$ and $b_x = \mu^x(2)$. 
Let $O_{a,b} =  
\{ x \in (\omega)^3 : a_x = a \wedge b_x = b \}$. Set   
$R = \{ x \in (\omega)^3 : f_A(a_x) \leq b_x \}$. Since  $R = \bigcup \{ O_{n,m} \mid f_A(n) \leq m \}$ and $\overline{R} = \bigcup \{ O_{n,m} \mid f_A(n) > m \}$
both $R$ and $\overline{R}$ are $A$-computable open sets. 

\begin{claim}
If $p \in (\omega)^{\omega}$ is homogeneous, then $(p)^3 \subseteq R$. 
\end{claim}

It suffices to show that there is an $x \in (p)^3$ with $x \in R$.  Let $u = \mu^p(1)$. 
Because $p$ has infinitely many blocks, there must be some $i$ with $\mu^p(i) \geq f(u)$. Consider the partition $x = w \circ p$, where $w(1) = 1, w(i) =2,$ and $w(m) = 0$ for all other $m$.  Then since $a_x =
u$ and $b_x \geq f(u)$, we have $x \in (p)^3$ with $f(a_x) \leq b_x$, so $x \in R$.

\begin{claim}
If $p \in (\omega)^{\omega}$ is homogeneous, then $A \leq_T p$.
\end{claim}

Fix $p$ and let $g(n) = \mu^p(n+2)$. Since $g$ is $p$-computable, it suffices to show $g \succeq f_A$. Because  
$n < \mu^p(n+1)$ and $f_A$ is increasing, we have $f_A(n) < f_A(\mu^p(n+1))$. Therefore, to show $g \succeq f_A$, it suffices 
to show $f_A(\mu^p(n+1)) \leq \mu^p(n+2) = g(n)$. 

Let $x_n \in (p)^3$ be defined by $x_n = w_n\circ p$, where
$w_n(n+1) = 1, w_n(n+2) = 2,$ and $w_n(m) = 0$ for all other $m$.
Note that $a_{x_n} = \mu^p(n+1)$ and $b_{x_n} =
\mu^p(n+2)$.  By the previous claim, $x_n \in R$, so
$f_A(a_{x_n}) \leq b_{x_n}$. In other words, $f_A(\mu^p(n+1)) \leq \mu^p(n+2)$ as required.
\end{proof}

\begin{cor}
For each $k\geq 3$ and each recursive ordinal
$\alpha$, there is an $\emptyset_{(\alpha)}$-computable clopen
set $R\subseteq (\omega)^k$ such that if $p \in
(\omega)^{\omega}$ is homogeneous for $(\omega)^k = R \cup \overline{R}$, 
then $\emptyset_{(\alpha)} \leq_T p$.
\end{cor}
\begin{proof} 
For $k=3$, this corollary follows from Theorems \ref{thm:fact} and \ref{thm:main}. For $k > 3$, use similiar definitions for 
$R$ and $\overline{R}$ ignoring what happens after the first three blocks of the partition. 
\end{proof}

\begin{thm}\label{thm:hyp}
For every recursive ordinal $\alpha$, and every $k\geq 3$, there is a computable 
Borel code for a $\Delta^0_{\alpha+1}$ set $R\subseteq (\omega)^k$ such that 
every $p \in (\omega)^\omega$
homogeneous for the coloring $(\omega)^k = R \cup \overline{R}$ computes $\emptyset_{(\alpha)}$.
\end{thm}

\begin{proof}
Let $R, \overline{R}$ 
be the $\emptyset_{(\alpha)}$-computable 
clopen sets from the previous corollary.  By Proposition \ref{prop:Borelcodes},
both $R$ and $\overline{R}$ have computable Borel codes as $\Sigma^0_{\alpha+1}$ 
subsets of $(\omega)^k$.  Therefore, $R$ has a computable Borel code
as $\Delta^0_{\alpha+1}$ set.  By the previous corollary, if $p$ is 
homogeneous for $(\omega)^k = R \cup \overline{R}$, then $p \geq_T \emptyset_{(\alpha)}$, as required.
\end{proof}

For $\alpha = 1$, Theorem \ref{thm:hyp} says there is a $\Delta^0_2$ clopen set $R \subseteq (\omega)^3$ such that $R$ and $\overline{R}$ have computable Borel 
codes as $\Sigma^0_2$ sets (and hence as $\Delta^0_2$ sets) and any homogeneous partition for $(\omega)^3 = R \cup \overline{R}$ computes $\emptyset'$. 

\section{The Borel Dual Ramsey Theorem for $k=2$}
\label{sec:4} 

\subsection{Effective Analysis} 
We consider the complexity of finding infinite homogeneous partitions for colorings $(\omega)^2 = R \cup \overline{R}$ as a function of the descriptive complexity of
$R$ and/or $\overline{R}$. 
We begin by showing that if $R$ is a computable open set, there is a computable homogeneous partition.

\begin{thm}
\label{thm:sigma1}
Let $R$ be a computable code for an open set in $(\omega)^2$.  
There is a computable $p \in (\omega)^{\omega}$ such that $(p)^2 \subseteq R$ or 
$(p)^2 \subseteq \overline{R}$.
\end{thm}

\begin{proof}
If there is an $n \geq 1$ such that $[0^n] \cap R = \emptyset$, then the partition $x \in (\omega)^{\omega}$ with blocks 
$\{0,1,\dots, n\}, \{n+1\}, \{n+2\}, \dots$ satisfies $(x)^2 \subseteq \overline{R}$.  Otherwise, for arbitrarily large $n$ 
there are $\tau \succ 0^n1$ with $[\tau]\subseteq R$, and hence  
there is a computable sequence $\tau_1,\tau_2,\dots$ of such $\tau$ with $0^i\prec \tau_i$.  Computably thin this sequence 
so that for each $i$, $0^{|\tau_i|} \prec \tau_{i+1}$.  The 
partition $x$ with blocks $\block{x}{i} = \{j : \tau_i(j) = 1\}$ for $i > 0$ satisfies $(x)^2 \subseteq R$. 
\end{proof}

To extend to sets coded at higher finite levels of the Borel hierarchy, we will need the following generalization of 
the previous result. 

\begin{thm}
\label{thm:sigma1_w_density}
Let $R$ be a computable code for an open set in $(\omega)^2$ such that 
$R \cap [0^n] \neq \emptyset$ for all $n$. Let $\{D_i\}_{i<\omega}$ be a uniform sequence of computable codes for 
open sets such that each $D_i$ is dense in $R$. There is a computable $x \in (\omega)^{\omega}$ such that 
$(x)^2 \subseteq R \cap (\cap_i D_i)$.
\end{thm}

\begin{proof}
We build $x$ as the limit of an effective sequence $\tau_0 \prec \tau_1 \prec \cdots$ with $\tau_s \in (<\omega)^{s+1}$. 
We define the strings $\tau_s$ in stages starting with $\tau_0 = \langle 0 \rangle$ which puts $x(0)=0$. For $s \geq 1$, we ensure that at the start of 
stage $s+1$, we have $[\sigma \circ \tau_s] \subseteq R$ for all $\sigma \in (s+1)^2$. 
That is, the open sets in $(\omega)^2$ determined by each way of coarsening the $s+1$ many blocks of $\tau_s$ to two blocks is contained in $R$. 

At stage $s+1$, assume we have defined $\tau_s \in (<\omega)^{s+1}$. If $s \geq 1$, assume that for all 
$\sigma \in (s+1)^2$, $[\sigma \circ \tau_s] \subseteq R$. Let $\sigma_0, \ldots, \sigma_{M_s-1}$ list the strings 
$\sigma \in (s+2)^2$. We define a sequence of strings $\tau_s^0 \prec\dots\prec \tau_s^{M_s}$ and set 
$\tau_{s+1} = \tau_s^{M_s}$. 

We define $\tau_s^0$ to start a new block as follows. 
Since $[0^{|\tau_s|}] \cap R \neq \emptyset$, we  
effectively search for $\gamma_s \in (<\omega)^2$ such that 
$0^{|\tau_s|} \prec \gamma_s$ and $[\gamma_s] \subseteq R$. Since 
$\gamma_s \in (<\omega)^2$, there is at least one $m < |\gamma_s|$ such that $\gamma_s(m) = 1$. Define $\tau_s^0$ with 
$|\tau_s^0| = |\gamma_s|$ by 
\[
\tau_s^0(m) = 
\left\{
\begin{array}{ll}
\tau_s(m) & \text{if } m < |\tau_s| \\
s+1 & \text{if } \gamma_s(m) = 1 \text{ (and hence } m \geq  |\tau_s| \text{)} \\
0 & \text{if } m \geq |\tau_s| \text{ and } \gamma_s(m) = 0.
\end{array}
\right.
\]
Note that $\tau_s \prec \tau_s^0$, and that $[\sigma \circ \tau_s^0] \subseteq R$ 
for all $\sigma \in (s+2)^2$.  To see the latter, let $j$ be least such that $\sigma(j)=1$
and consider two cases.  If $j<s+1$, then $\sigma\restrict s+1 \in (s+1)^2$ 
and the conclusion follows by the induction hypothesis.  If $j=s+1$, 
then $\sigma \circ \tau_s^0 = \gamma_s$.

We continue to define the $\tau_s^j$ strings by induction. Assume that $\tau_s^j$ has been defined and consider the $j$-th string $\sigma_j$ 
enumerated above describing how to collapse $(s+2)$ many blocks into 2 blocks. Since $\tau_s^0 \preceq \tau_s^j$, we have 
$\sigma_j \circ \tau_s^0 \preceq \sigma_j \circ \tau_s^j$ and hence $[\sigma_j \circ \tau_s^j] \subseteq R$. Because $\cap_{n < s+1} D_n$ is dense in $R$, 
we can effectively search for a string $\delta_s^j \in (<\omega)^2$ such that $\sigma_j \circ \tau_s^j \preceq \delta_s^j$ and $[\delta_s^j] \subseteq 
\cap_{n < s+1} D_n$. To define $\tau_s^{j+1}$, we uncollapse $\delta_s^j$. Let $j^*$ be the least number such that $\sigma_j(j^*) = 1$. Define 
\[
\tau_s^{j+1}(m) = 
\left\{
\begin{array}{ll}
\tau_s^j(m) & \text{if } m < |\tau_s^j| \\
j^* & \text{if } m \geq |\tau_s^j| \text{ and } \delta_s^j(m) = 1 \\
0 & \text{if } m \geq |\tau_s^j| \text{ and } \delta_s^j(m) = 0
\end{array}
\right.
\]
It is straightforward to check that $\tau_s^j \preceq \tau_s^{j+1}$ and that $\sigma_j \circ \tau_s^{j+1} = \delta_s^j$. This completes the construction of 
the sequence $\tau_s^0 \preceq \cdots \preceq \tau_s^{M_s}$ and of the computable partition $x$. It remains to show that if $p \in (x)^2$, then $p \in R$ and 
$p \in \cap_{n \in \omega} D_n$. Fix $p \in (x)^2$ and let $w \in (\omega)^\omega$
be such that $\omega\circ x = p$.  Let $s_0$ be least such that $w(s_0+1)=1$.

\begin{claim}
$p \in R$.
\end{claim}

Let $\sigma = (0^{s_0})^\smallfrown 1$, so that $\sigma \prec w$.
At stage $s_0+1$, we defined $\tau_{s_0}^0 \prec x$ with the property that 
$[\sigma \circ \tau_{s_0}^0] \subseteq R$. Since $\sigma \circ \tau_s^0 \prec p$, we have $p \in R$. 

\begin{claim}
$p \in \cap_{n < \omega} D_n$.
\end{claim}

Fix $k \in \omega$ and we show $p \in D_k$. Let $s = \max \{ k,s_0 \}$. Consider the action during stage $s+1$ of the construction. Let $\sigma = w \restrict(s+2)$.
Then $\sigma\in(s+2)^2$, so let $j$ be such that $\sigma_j = \sigma$.
 We defined $\delta_s^j$ and $\tau_s^{j+1}$ such that 
$\sigma_j \circ \tau_s^{j+1} = \delta_s^j$ and $[\delta_s^j] \subseteq \cap_{n < s+1} D_n$, so in particular, $[\delta_j^s] \subseteq D_k$. Since $\tau_s^{j+1} \prec x$, we have $\delta_s^j = \sigma_j \circ \tau_s^{j+1} \prec p$, so $p \in D_k$ as required. 
\end{proof}

The next proposition is standard, but we present the proof because some details will be relevant to Theorem \ref{thm:lifting}. In the proof, we use 
codes for open sets as in Definition \ref{defn:reverseopen}. 

\begin{prop}
\label{prop:baire_degrees}
Let $n \in \omega$ and let $A \subseteq 2^{\omega}$ be defined by a $\Sigma^0_{n+1}$ predicate. There are a $\Delta^0_{n+1}$ code $U$ for an open 
set in $(\omega)^2$, a $\Delta^0_{n+2}$ code $V$ for an open set in $(\omega)^2$ and a uniformly $\Delta^0_{n+1}$ sequence $\langle D_i : i \in \omega \rangle$ of codes for 
dense open sets such that $U \cup V$ is dense and for all $p \in \cap_{i \in \omega} D_i$, if $p \in U$, then $p \in A$ and if $p \in V$ then 
$p \not \in A$. Furthermore, the $\Delta^0_{n+1}$ and $\Delta^0_{n+2}$ indices for $U$, $V$ and $\langle D_i : i \in \omega \rangle$ can be obtained uniformly from a 
$\Sigma^0_{n+1}$ index for $A$. 
\end{prop}

\begin{proof}
We proceed by induction on $n$. Throughout this proof, $\sigma$, $\tau$, $\rho$ and $\delta$ denote elements of $(<\omega)^2$. In addition to the 
properties stated in the proposition, we ensure that if $\langle m,\sigma \rangle \in U$ (or $V$) and $\tau \succeq \sigma$, then there is a $k$ such 
that $\langle k,\tau \rangle \in U$ (or $V$ respectively). Thus, if $U \cap [\sigma] \neq \emptyset$, then there is $\langle k, \tau \rangle \in U$ with $\sigma \preceq \tau$.

For $n=0$, we have $X \in A \Leftrightarrow \exists k \, \exists m \, P(m, X \upharpoonright k)$ where $P(x,y)$ is a $\Pi^0_0$ 
predicate. Without loss of generality, we assume that if $P(m, X \upharpoonright k)$ holds, then $P(m', Y \upharpoonright k')$ holds for all $k' \geq k$, $m' \geq m$ 
and $Y \in 2^{\omega}$ such that $Y \upharpoonright k = X \upharpoonright k$. Let 
$U = \{ \langle n,\sigma \rangle : P(\sigma,n) \}$, $V = \{ \langle 0,\sigma \rangle : \forall x \, \forall \tau \succeq \sigma \, (\neg P(\tau,x)) \}$ and 
$D_i = (<\omega)^2$ for $i \in \omega$. It is straightforward to check these codes have the required properties. 

For the induction case, let $A \subseteq 2^{\omega}$ be defined by a $\Sigma^0_{n+2}$ predicate, so $X \in A \Leftrightarrow \exists k P(X,k)$ where $P$ is a 
$\Pi^0_{n+1}$ predicate. For $k \in \omega$, let $A_k = \{ X : \neg P(X,k) \}$. Apply the induction hypothesis to $A_k$ to  
fix indices (uniformly in $k$) for the $\Delta^0_{n+1}$ codes $U_k$ and $\langle  D_{i,k} : i \in \omega \rangle$ and for the $\Delta^0_{n+2}$ code $V_k$ 
so that if $p \in \cap_{i \in \omega} D_{i,k}$, then $p \in U_k$ implies $\neg P(k,p)$ and $p \in V_k$ implies $P(k,p)$. Let 
\begin{gather*}
U = \{ \langle \langle k,m \rangle, \sigma \rangle : \langle m,\sigma \rangle \in V_k \} \text{ and} \\
V = \{ \langle 0,\sigma \rangle : \forall k \, \forall \tau \succeq \sigma \, \exists m \, \exists \rho \succeq \tau \, \langle m,\rho \rangle \in U_k \}.
\end{gather*}
$U$ is a $\Delta^0_{n+2}$ code for $\cup_k V_k$, and $V$ is a $\Delta^0_{n+3}$ code such that $\langle m, \sigma \rangle \in V$ if and only if every 
$U_k$ is dense in $[\sigma]$. 
We claim that $U \cup V$ is dense. Fix $\sigma$ and assume $U \cap [\sigma] = \emptyset$, so $V_k \cap [\sigma] = \emptyset$ for all $k$. 
Since $U_k \cup V_k$ is dense, $U_k \cap [\tau] \neq \emptyset$ for all $\tau \succeq \sigma$ and all $k$, so $\langle 0, \sigma \rangle \in V$. 

For $i = \langle a_i,b_i \rangle$, define $D_i = D_{a_i,b_i} \cap (U_i \cup V_i)$. $D_i$ has a $\Delta^0_{n+2}$ code as a dense open set and the index can be uniformly computed 
from the indices for $U_i$, $V_i$ and $D_{a_i,b_i}$. Furthermore, if $p \in \cap_i D_i$ then $p \in \cap_{i,k} D_{i,k}$ and $p \in \cap_k (U_k \cup V_k)$. 

Assume that $p \in \cap_i D_i$. First, we show that if $p \in U$, then $p \in A$. Suppose $p \in U = \cup_k V_k$ and fix $k$ such that $p \in V_k$. 
Since $p \in \cap_i D_{i,k}$ for this fixed $k$, $p \not \in A_k$ by the induction hypothesis. Therefore, $P(k,p)$ holds and hence $p \in A$. 

Second, we show that if $p \in V$ then $p \not \in A$. Assume $p \in V$ and fix $\langle 0, \sigma \rangle \in V$ such that $\sigma \prec p$. It suffices to show $\neg P(k,p)$ holds 
for an arbitrary $k \in \omega$. Since $p \in \cap_i D_i$, we have $p \in U_k \cup V_k$ and $p \in \cap_i D_{i,k}$. If $p \in U_k$, then $\neg P(k,p)$ holds by induction and we 
are done. Therefore, suppose for a contradiction that $p \in V_k$. Fix $\langle 0, \tau \rangle \in V_k$ such that $\sigma \preceq \tau$ and $\tau \prec p$. 
Since $\langle 0,\sigma \rangle \in V$ and $\sigma \preceq \tau$, there are $\rho \succeq \tau$ and $m$ such that 
$\langle m,\rho \rangle \in U_k$, and therefore $[\rho] \subseteq U_k \cap V_k$. This containment is the desired contradiction because $q \in [\rho] \cap \cap_i D_{i,k}$ would  
satisfy $q \in A_k$ and $q \not \in A_k$.  
\end{proof}

\begin{thm}
\label{thm:lifting}
For every coloring $(\omega)^2 = R\cup \overline{R}$ such that $R$ is a computable code for a 
$\Sigma_{n+2}^0$ set, there is either a $\emptyset^{(n)}$-computable $x \in (\omega)^{\omega}$ which is homogeneous for $\overline{R}$ or a 
$\emptyset^{(n+1)}$-computable $x \in (\omega)^{\omega}$ which is homogeneous for $R$.
\end{thm}

\begin{proof}
Fix $R$ and fix a $\Pi^0_{n+1}$ predicate $P(k,y)$ such that for $y \in (\omega)^2$, 
$y \in R \Leftrightarrow \exists k \, P(k,y)$. Let $U_k$, $V_k$ and 
$\langle D_{i,k} : i \in \omega \rangle$ be the codes from Proposition \ref{prop:baire_degrees} for $R_k = \{ y : \neg P(y,k) \}$. Let 
$U = \cup_k V_k$, $V = \cup \{ [\sigma] : \forall k \ U_k \text{ is dense in } [\sigma]\}$ and $D_i$, $i \in \omega$, be the corresponding 
codes for $R$. We split non-uniformly into cases.

{\bf Case 1:} Assume $V$ is dense in $[0^\ell]$ for some fixed $\ell$. We make two observations. First, $U$ is disjoint from $[0^\ell]$. Therefore, each 
$V_k$ is disjoint from $[0^\ell]$ and hence each $U_k$ is dense in $[0^\ell]$. Second, suppose $y \in (\bigcap_{i,k} D_{i,k}) \cap (\bigcap_{k} U_k)$. For each $k$ we have 
$y \in \cap_i D_{i,k}$ and $y \in U_k$, so $\forall k \, \neg P(k,y)$ holds and hence $y \in \overline{R}$.

We apply Theorem \ref{thm:sigma1_w_density} relativized to $\emptyset^{(n)}$ to the computable open set $O = [0^\ell]$ (which has nonempty intersection with $[0^j]$ for 
every $j$) and the $\emptyset^{(n)}$-computable sequence of codes $D_{i,k}$ and $U_k$ for $i,k < \omega$. By the first observation, each coded set in this sequence is dense in 
$O$. Therefore, there is a $\emptyset^{(n)}$-computable $x \in (\omega)^{\omega}$ such that $(x)^2 \subseteq [0^{\ell}] \cap (\bigcap_{i,k} D_{i,k}) \cap 
(\bigcap_{k} U_k)$. By the second observation, $(x)^2 \subseteq \overline{R}$ as required. 
 
 {\bf Case 2:} Assume $V$ is not dense in $[0^m]$ for any $m$. In this case, since $U \cup V$ is dense, we have $U \cap [0^m] \neq \emptyset$ for all $m$. 
 We apply Theorem \ref{thm:sigma1_w_density} relativized to $\emptyset^{(n+1)}$ to the $\emptyset^{(n+1)}$-computable open set $U$ and the 
 $\emptyset^{(n+1)}$-computable sequence of dense sets $D_i$ for $i \in \omega$ to obtain an $\emptyset^{(n+1)}$-computable $x$ with 
 $(x)^2 \subseteq U \cap (\bigcap_{i} D_i) \subseteq R$ as required.  
\end{proof}

We end this section by showing that the non-uniformity in the proof of Theorem \ref{thm:sigma1} is necessary.

\begin{thm}
\label{thm:non_uniform}
For every Turing functional $\Delta$, there are computable codes $R_0$ and $R_1$ for complementary open sets in $(\omega)^2$ such
that $\Delta^{R_0 \oplus R_1}$ is not an infinite homogeneous
partition for the reduced coloring $(\omega)^2 = R_0 \cup R_1$.
\end{thm}

\begin{proof}
Fix $\Delta$. We define $R_0$ and $R_1$ in stages as $R_{0,s}$ and
$R_{1,s}$.  Our construction proceeds in a basic module while we wait
for $\Delta_s^{R_{0,s} \oplus R_{1,s}}$ to provide appropriate
computations. If these computations appear, we immediately
diagonalize and complete the construction.

For the basic module at stage $s$, put
$0^{2s+1}1 \in R_{0,s}$ and $0^{2s+2}1 \in R_{1,s}$.
Check whether there is a $0 < k < s$ such that $\Delta_{s}^{R_{0,s}
  \oplus R_{1,s}}(i) = 0$ for all $i < k$ and $\Delta_{s}^{R_{0,s}
  \oplus R_{1,s}}(k) = 1$. If there is no such $k$, then we proceed to
stage $s+1$ and continue with the basic module.

If there is such a $k$, then we stop the basic module and fix $i < 2$
such that $0^k1 \in R_{i,s}$. (Since $k < s$, we have already
enumerated $0^k1$ into one of $R_{0,s}$ or $R_{1,s}$ depending on
whether $k$ is even or odd.) We end the construction at this stage and define $R_i =
R_{i,s}$ and $R_{1-i} = R_{1-i,s} \cup \{ 0^t1 \mid 2s+2 < t \}$.

This completes the construction. It is clear that $R_0$ and $R_1$ are computable codes
for complementary open sets and $(\omega)^2 = R_0 \cup R_1$ is a reduced coloring. If the
construction never finds an appropriate value $k$, then $\Delta^{R_0 \oplus
  R_1}$ is not an element of $(\omega)^\omega$ and we are
done. Therefore, assume we find an appropriate value $k$ at stage
$s$ in the construction.  Fix $i$ such that $0^k1 \in R_{i,s}$ and
assume that $p = \Delta^{R_0 \oplus R_1}$ is a element of
$(\omega)^\omega$. We show $p$ is not homogeneous by giving elements $q_0, q_1 \in (p)^2$ such that
$q_0 \in R_i$ and $q_1 \in R_{1-i}$.

By construction, $0^k 1 \prec p$. Let $q_0
  \in (p)^2$ be any coarsening with $0^k1 \prec q_0$
  Then $q_0 \in R_i$ because $[0^k1] \subseteq R_i$.

On the other hand, since $p \in (\omega)^{\omega}$, there are
infinitely many $p$-blocks. Let $n$ be least with $\mu^p(n) > 2s+2$. Let $q_1 \in (p)^2$ be any coarsening for which 
$q_1 \in [0^{\mu^p(n)}1]$. Since $\mu^p(n) > 2s+2$, we put $0^{\mu^p(n)}1 \in R_{1-i}$, so $q_1 \in R_{1-i}$ 
as required. 
\end{proof}

\subsection{Strong reductions for reduced colorings}\label{sec:weihrauch}

In this section, we think of $\BoDRT^2_2$ as an \emph{instance-solution problem}. Such a problem consists of a collection of subsets of 
$\omega$ called the \emph{instances} of this problem, and for each instance, a collection of subsets of $\omega$ called the \emph{solutions} to this instance 
(for this problem). A problem $\mathsf{P}$ is \emph{strongly Weihrauch reducible} to a problem $\mathsf{Q}$ if there are fixed Turing functionals $\Phi$ and $\Psi$ such 
that given any instance $A$ of $\mathsf{P}$, $\Phi^A$ is an instance of $\mathsf{Q}$, and given any solution $B$ to $\Phi^A$ in $\mathsf{Q}$, $\Psi^B$ is a solution to 
$A$ in $\mathsf{P}$. There are a number of variations on this reducibility and we refer to the reader to \cite{Dzhafarov-2015ta} and \cite{HJ-2015ta} for background on 
these reductions and for connections to reverse mathematics. In this paper, we will only be interested in problems arising out of $\Pi^1_2$ statements of second order 
arithmetic. Any such statement can be put in the form $\forall X(\varphi(X) \to \exists Y \psi(X,Y))$, where $\phi$ and $\psi$ are arithmetical. We can then regard this as a 
problem, with instances being all $X$ such that $\varphi(X)$, and the solutions to $X$ being all $Y$ such that $\psi(X,Y)$. Note that while the choice of $\varphi$ and $\psi$ 
is not unique, we always have a fixed such choice in mind for a given $\Pi^1_2$ statement, and so also a fixed assignment of instances and solutions.

A reduced coloring $(\omega)^2 = R_0 \cup R_1$ is classically open and the color of $p \in (\omega)^2$ depends only on $\mu^p(1)$. When $R_0$ and $R_1$ are 
codes for open sets, there is a homogeneous partition computable in $R_0 \oplus R_1$, although by Theorem \ref{thm:non_uniform}, 
not uniformly. We consider the case when the open sets $R_0$ and $R_1$ are represented by Borel codes for $\Sigma^0_n$ sets with $n \geq 2$. 

$\Delta^0_n$-$\mathsf{rDRT}^2_2$ is the statement that for each reduced coloring $(\omega)^2 = R_0 \cup R_1$ where $R_0$ and $R_1$ are Borel codes 
for $\Sigma^0_n$ sets, there exists an $x \in (\omega)^\omega$ and an $i < 2$ such that $(x)^2 \subseteq R_i$. In effective algebra, this statement is 
clear, but in $\RCA$, we need to specify how to handle these codes. 

Recall that a Borel code for a $\mathbf \Sigma^0_n$ set is a labeled subtree of $\omega^{< n+1}$ which we write as $(B,\varphi)$ to specify the labeling function $\varphi$. 
For a leaf $\sigma$ and a partition $p$, we write $p \in \varphi(\sigma)$ if $p$ is an element of the clopen set coded by $\varphi(\sigma)$, and we write $\varphi(\sigma) = [\tau]$ to avoid specifying a coding scheme. 

In reverse mathematics there are two ways that membership in a $\mathbf \Sigma^0_\alpha$
set could be discussed.  The \emph{evaluation map} method works for 
arbitrary $\alpha$ and requires a strong base theory.  This method will 
be discussed in the next section.  The \emph{virtual}
method works only for finite $\alpha$.  For each $n<\omega$, there is a 
fixed $\Sigma^0_n$ formula $\eta(B,\varphi, p)$ such that if $(B,\varphi)$ is a Borel code for a $\mathbf \Sigma^0_n$ set and $p \in (\omega)^2$, then $\eta(B,\varphi,p)$ says $p$ is in the 
set coded by $(B,\varphi)$.  In this section we use only the virtual method.

The formula is defined as follows.  We begin by defining formulas $\beta_k(\sigma, B, \varphi, p)$ for $1 \leq k \leq n$ by downward induction on $k$. For $\sigma \in B$ with $|\sigma| = k$, 
$\beta_k(\sigma, B, \varphi, p)$ says that $p$ is in the set coded by the labeled subtree of $(B,\varphi)$ above $\sigma$. Since any $\sigma \in B$ 
with $|\sigma| = n$ is a leaf, $\beta_n(\sigma, B, \varphi, p)$ is the formula $p \in \varphi(\sigma)$. For $1 \leq k < n$, $\beta_k(\sigma, B, \varphi, p)$ is the formula  
\begin{gather*}
(\varphi(\sigma) = \cup \rightarrow \alpha_{k}^{\cup}) \wedge (\varphi(\sigma) = \cap \rightarrow \alpha_{k}^{\cap}) \wedge (\varphi(\sigma) \in L \rightarrow \alpha_{k}^L), \text{ where} \\
\alpha_{k}^{\cup}(\sigma,B, \varphi, p) \text{ is } \exists \tau \in B \big( \sigma \prec \tau \wedge |\tau|=k+1 \wedge \beta_{k+1}(\tau,B,\varphi,p) \big) \\
\alpha_{k}^{\cap}(\sigma, B, \varphi, p) \text{ is } \forall \tau \in B \big( (\sigma \prec \tau \wedge |\tau|=k+1) \rightarrow \beta_{k+1}(\tau,B,\varphi,p) \big) \\
\text{and }\alpha_{k}^L(\sigma, B, \varphi,p) \text{ is } p \in \varphi(\sigma).
\end{gather*}
The formula $\eta(B, \varphi, p)$ is $\exists \sigma \in B \, (|\sigma| = 1 \wedge \beta_1(\sigma, B, \varphi, p))$. In $\RCA$, we write $p \in B$ for $\eta(B,\varphi,p)$. The statement 
$\Delta^0_n$-$\mathsf{rDRT^2_2}$ now has the obvious translation in $\RCA$.

A Borel code $(B, \varphi)$ for a $\mathbf \Sigma^0_n$ set is \textit{in normal form} if $B = \omega^{< n+1}$ and for every $\sigma$ 
with $|\sigma| < n$, if $|\sigma|$ is even, then $\varphi(\sigma) = \cup$, and if $|\sigma|$ is odd, then $\varphi(\sigma) = \cap$. In $\RCA$, for every $(B,\varphi)$, there is a 
$(\widehat{B}, \widehat{\varphi})$ in normal form such that for all $p \in (\omega)^2$, $p \in B$ if and only if $p \in \widehat{B}$. Moreover, 
the transformation from $(B,\varphi)$ to $(\widehat{B}, \widehat{\varphi})$ is uniformly computable in $(B,\varphi)$. We describe the transformation when $(B,\varphi)$ is a Borel code for a 
$\Sigma^0_2$ set. The case for a $\Sigma^0_n$ set is similar. 

Let $(B,\varphi)$ be a Borel code for a $\Sigma^0_2$ set. By definition, $\lambda \in B$ with $\varphi(\lambda) = \cup$. Each $\sigma \in B$ with $|\sigma| = 1$ is the root of a subtree 
coding a $\Sigma^0_0$ set (if $\varphi(\sigma) \in L$), a $\Sigma^0_1$ set (if $\varphi(\sigma) = \cup$) or a $\Pi^0_1$ set (if $\varphi(\sigma) = \cap$). 
Consider the following sequence of transformations. 
\begin{itemize}
\item To form $(B_1, \varphi_1)$, for each $\sigma \in B$ with $|\sigma| = 1$ and 
$\varphi(\sigma) = \cup$, remove the subtree of $B$ above $\sigma$ (including $\sigma$). For each $\tau \in B$ with $\tau \succ \sigma$, add a new node $\tau'$ to $B_1$ with $|\tau'|=1$ and 
$\varphi_1(\tau') = \varphi(\tau) \in L$. 

\item To form $(B_2,\varphi_2)$, for each leaf $\sigma \in B_1$ with $|\sigma| = 1$, relabel $\sigma$ by $\varphi_2(\sigma) = \cap$ and add a new successor $\tau$ to $\sigma$ with 
label $\varphi_2(\tau) = \varphi_1(\sigma)$. 

\item To form $(B_3, \varphi_3)$, for each $\sigma \in B_2$ with $|\sigma| = 1$, let $\tau_{\sigma} \in B_1$ be the first successor of $\sigma$. Add infinite many new nodes 
$\delta \succ \sigma$ to $B_3$ with $\varphi_3(\delta) = \varphi_2(\tau_{\sigma})$. 

\item To form $(B_4, \varphi_4)$, let $\sigma$ be the first node of $B_3$ at level 1. Add infinitely many copies of the subtree above $\sigma$ to $B_4$ with 
the same labels as in $B_3$. 
\end{itemize}
In $(B_4,\varphi_4)$, the leaves are at level 2, every interior node is infinitely branching and $\varphi_4(\sigma) = \cap$ when $|\sigma| = 1$. 
There is a uniform procedure to define a bijection $f:B_4 \rightarrow \omega^{< 3}$. We define $(\widehat{B}, \widehat{\varphi})$ by $\widehat{B} = \omega^{< 3}$ and 
$\widehat{\varphi}(\sigma) = \varphi_4(f^{-1}(\sigma))$. In $\RCA$, for all $p \in (\omega)^2$, 
$\eta(B,\varphi,p)$ holds if and only if $\eta(\widehat{B}, \widehat{\varphi},p)$ holds.

When $(B,\varphi)$ is a Borel code for a $\Sigma^0_n$ set in normal form, $\eta(B,\varphi,p)$ is equivalent to 
$\exists x_0 \, \forall x_1 \cdots \mathsf{Q}_{n-1} x_{n-1} \, (p \in \varphi(\langle x_0, x_1, \ldots, x_{n-1} \rangle))$ where $\mathsf{Q}_{n-1}$ is $\forall$ or $\exists$ depending on whether  
$n-1$ is odd or even. We have analogous definitions for Borel codes for $\Pi^0_n$ sets in normal form. 

To define $\mathsf{D}^n_2$, let $[\omega]^n$ denote the set of $n$ element subsets of $\omega$. We view the elements of $[\omega]^n$ as strictly increasing sequences 
$s_0 < s_1 < \cdots < s_{n-1}$. 

\begin{defn}
A coloring $c:[\omega]^n \rightarrow 2$ is \textit{stable} if for all $k$, the limit 
\[
\lim_{s_1} \cdots \lim_{s_{n-1}}~c(k,s_1,\ldots,s_{n-1})
\]
exists. $L \subseteq \omega$ is \textit{limit-homogeneous} for a stable coloring $c$ if there is an $i < 2$ such that for each $k \in L$, 
\[
\lim_{s_1} \cdots \lim_{s_{n-1}}~c(k,s_1,\ldots,s_{n-1}) = i.
\]
$\mathsf{D}^n_2$ is the statement that each stable coloring $c : [\omega]^n \to 2$ has an infinite limit-homogeneous set. 
\end{defn}

Below, the proof of Theorem \ref{thm:sW}(2) is a formalization of the proof of Theorem \ref{thm:sW}(1), and the additional induction used is a consequence of this formalization. 
We do not know if its use is necessary; that is, we do now if $\RCA + \mathsf{I}\Sigma^0_{n-1}$ can be replaced simply by $\RCA$ when $n > 2$. 

\begin{thm}
\label{thm:sW}
Fix $n \geq 2$.
	\begin{enumerate}
		\item $\Delta^0_n$-$\mathsf{rDRT}^2_2 \equiv_{\mathrm{sW}} \mathsf{D}^n_2$.
		\item Over $\RCA + \mathsf{I}\Sigma^0_{n-1}$, $\Delta^0_n$-$\mathsf{rDRT}^2_2$ is equivalent to $\mathsf{D}^n_2$. 
	\end{enumerate}
\end{thm}

\begin{cor}
$\Delta^0_2$-$\mathsf{rDRT}^2_2$ is equivalent to $\mathsf{SRT}^2_2$ over $\RCA$.
\end{cor}

\begin{proof}
$\mathsf{D}^2_2$ is equivalent to $\mathsf{SRT}^2_2$ over $\RCA$ by Chong, Lempp, and Yang \cite{ChongLemppYang}.
\end{proof}

\begin{cor}
$\Delta^0_2$-$\mathsf{rDRT}^2_2 <_{\mathrm{sW}} \mathsf{SRT}^2_2$.
\end{cor}

\begin{proof}
$\mathsf{D}^2_2 <_{\mathrm{sW}} \SRT^2_2$ by Dzhafarov \cite[Corollary 3.3]{Dzhafarov-2015ta}. (It also 
follows immediately that $\Delta^0_2$-$\mathsf{rDRT}^2_2 \equiv_{\mathrm{W}} \mathsf{D}^2_2 <_{\mathrm{W}} \mathsf{SRT}^2_2$.)
\end{proof}

\begin{proof}[Proof of Theorem \ref{thm:sW}]
We prove the two parts simultaneously, remarking, where needed, how to formalize the argument in $\RCA+\mathsf{I}\Sigma^0_{n-1}$.
	
To show that $\Delta^0_n$-$\mathsf{rDRT}^2_2 \leq_{\mathrm{sW}} \mathsf{D}^n_2$, and that $\Delta^0_n$-$\mathsf{rDRT}^2_2$ is implied by $\mathsf{D}^n_2$ over 
$\mathsf{RCA}_0+\mathsf{I}\Sigma^0_{n-1}$, fix an instance $(\omega)^2 = R_0 \cup R_1$ of $\Delta^0_n$-$\mathsf{rDRT}^2_2$ where each $R_i$ is a Borel code for a 
$\Sigma^0_n$ set. Without loss of generality, $R_0$ and $R_1$ are in normal form. For each $k \geq 1$, fix the partition $p_k = \chi_{\{k\}}$ (that is, $p_k$ has blocks $\omega\setminus \{k\}$ and $\{k\}$).
 
For $m < n$, we let $R_i(t_0, \ldots, t_m)$ denote the Borel set coded by the subtree of $R_i$ above $\langle t_0, \ldots, t_m \rangle$. Since 
$\langle t_0, \ldots, t_{n-1} \rangle$ is a leaf, $R_i(t_0, \ldots, t_{n-1})$ is the clopen set $\varphi_i(\langle t_0, \ldots, t_{n-1} \rangle)$. If $m < n-1$, then 
$R_i(t_0, \ldots, t_m)$ is a code for a $\Sigma^0_{n-(m+1)}$ set (if $m$ is odd) or a $\Pi^0_{n-(m+1)}$ set (if $m$ is even) in normal form. 

We define a coloring $c : [\omega]^n \to 2$ as follows. Let $c(0, s_1, \ldots, s_{n-1}) = 0$ for all $s_1 < \cdots < s_{n-1}$. For $m \leq n$, let $\mathsf{Q}_m$ 
stand for $\exists$ or $\forall$, depending as $m$ is even or odd, respectively. Given $1 \leq k < s_1 < \ldots < s_{n-1}$, define
	\[
		c(k,s_1,\ldots,s_{n-1}) = 1
	\]
	if and only if there is a $t_0 \leq s_1$ such that 
	\[
	(\forall t_1 \leq s_1) \cdots (\mathsf{Q}_m t_m \leq s_m) \cdots (\mathsf{Q}_{n-1} t_{n-1} \leq s_{n-1})~p_k \in \varphi_0(\langle t_0, \ldots, t_{n-1} \rangle)
	\]
	and for which there is no $u_0 < t_0$ such that 
	\[
	(\forall u_1 \leq s_1) \cdots (\mathsf{Q}_m u_m \leq s_m) \cdots (\mathsf{Q}_{n-1} u_{n-1} \leq s_{n-1})~p_k \in \varphi_1(\langle u_0, \ldots u_{n-1} \rangle).
	\]
(Note that $s_1$ bounds $t_0$, $t_1$ and $u_1$, whereas the other $s_m$ bound only $t_m$ and $u_m$.) The coloring $c$ is uniformly computable in $(R_0, \varphi_0)$ 
and $(R_1,\varphi_1)$ and is definable in $\RCA$ as a total function since all the quantification is bounded.
	
We claim that for each $k \geq 1$,
	\[
		\lim_{s_1} \cdots \lim_{s_{n-1}} c(k,s_1,\ldots, s_{n-1})
	\]
exists. Furthermore, if this limit equals $1$, then $p_k \in R_0$, and if this limit equals $0$, then $p_k \in R_1$. We break this claim into two halves.
	
First, for $1 \leq m \leq n-1$, we claim that for all fixed $1 \leq k < s_1 < \ldots < s_m$,
	\[
		\lim_{s_{m+1}} \cdots \lim_{s_{n-1}} c(k,s_1,\ldots,s_m,s_{m+1},\ldots,s_{n-1})
	\]
exists, and the limit equals $1$ if and only if there is a $t_0 \leq s_1$ such that 
	\begin{equation}\label{E:simeq_pos}
		(\forall t_1 \leq s_1) \cdots (\mathsf{Q}_m t_m \leq s_m)~p_k \in R_0(t_0,\ldots,t_m)
	\end{equation}
and there is no $u_0 < t_0$ such that 
	\begin{equation}\label{E:simeq_neg}
		(\forall u_1 \leq s_1) \cdots (\mathsf{Q}_m u_m \leq s_m)~p_k \in R_1(u_0,\ldots,u_m).
	\end{equation}
The proof is by downward induction on $m$. (In $\RCA$, the induction is performed externally, so we do not need to consider its complexity.) 
For $m = n-1$, there are no limits involved and the values of $c$ are correct by definition. 

Assume the result is true for $m+1$ and we show it remains true for $m$. By the definition of $R_0(t_0, \ldots, t_m)$, $t_0$ satisfies \eqref{E:simeq_pos} if and only if 
	\[
	(\forall t_1 \leq s_1) \cdots (\mathsf{Q}_m t_m \leq s_m)(\mathsf{Q}_{m+1} t_{m+1})~p_k \in R_0(t_0,\ldots,t_m,t_{m+1}),
	\]
	which in turn holds if and only if there is a bound $v$ such that for all $s_{m+1} \geq v$,
	\[
	(\forall t_1 \leq s_1) \cdots (\mathsf{Q}_m t_m \leq s_m)(\mathsf{Q}_{m+1} t_{m+1} \leq s_{m+1})~p_k \in R_0(t_0,\ldots,t_m,t_{m+1}).
	\]
If $\mathsf{Q}_{m+1}$ is $\exists$, then over $\RCA$, this equivalence requires a bounding principle. Since $p_k \in R_0(t_0, \ldots, t_{m+1})$ 
is a $\Pi^0_{n-(m+2)}$ predicate and $m+2 \geq 3$, we need at most $\mathsf{B}\Pi^0_{n-3}$ which follows from $\mathsf{I}\Sigma^0_{n-1}$. An analogous 
analysis applies to numbers $u_0$ satisfying \eqref{E:simeq_neg}. Thus, we can fix a common bound $v$ that works for all $t_0 \leq s_1$ in \eqref{E:simeq_pos} and 
all $u_0 < t_0 \leq s_1$ in \eqref{E:simeq_neg}.

Suppose there is a $t_0 \leq s_1$ satisfying \eqref{E:simeq_pos} for which there is no $u_0 < t_0$ satisfying \eqref{E:simeq_neg}. Then, for all $s_{m+1} \geq v$, $t_0$ 
satisfies the version of \eqref{E:simeq_pos} for $m+1$, and there is no $u_0 < t_0$ satisfying the version of \eqref{E:simeq_neg} for $m+1$. Therefore, by induction 
\[
\exists v \forall s_{m+1} \geq v \big( \lim_{s_{m+2}} \cdots \lim_{s_{n-1}} c(k,s_1, \ldots, s_{n-1}) = 1 \big)
\]
and hence $\lim_{s_{m+1}} \cdots \lim_{s_{n-1}} c(k,s_1, \ldots, s_{n-1}) = 1$ as required. 

On the other hand, suppose that there is no $t_0 \leq s_1$ satisfying \eqref{E:simeq_pos}, or that for every $t_0 \leq s_1$ satisfying \eqref{E:simeq_pos}, there is a $u_0 < t_0$ 
satsifying \eqref{E:simeq_neg}. Then, for all $s_{m+1} \geq v$, we have the analogous condition for $m+1$ and the induction hypothesis gives 
$\lim_{s_{m+1}} \cdots \lim_{s_{n-1}} c(k,s_1, \ldots, s_{n-1}) = 0$. This completes the first part of the claim. 
	
We can now prove the rest of the claim. For each $k \geq 1$, we have $p_k \in R_0$ or $p_k \in R_1$. Let $t_0$ be least such that 
$p_k \in R_0(t_0)$ or $p_k \in R_1(t_0)$. Since $p_k \in R_i(t)$ is a $\Pi^0_{n-1}$ statement, we use $\mathsf{I}\Sigma^0_{n-1}$ to 
fix this value in $\RCA$. 

Suppose $p_k \in R_0(t_0)$, so for all $u_0 < t_0$, it is not the case that $p_k \in R_1(u_0)$. By the first half of the claim with $m=1$, we have for every $s_1 \geq t_0$
	\[
		\lim_{s_2} \cdots \lim_{s_{n-1}} c(k,s_1,s_2,\ldots,s_{n-1}) = 1,
	\]
and therefore $\lim_{s_1} \cdots \lim_{s_{n-1}} c(k,s_1,\ldots,s_{n-1}) = 1$. 

Suppose $p_k \not \in R_0(t_0)$, and hence $p_k \in R_1(t_0)$. Again, by the first half of the claim with $m=1$, we have for every $s_1 \geq t_0$
	\[
		\lim_{s_2} \cdots \lim_{s_{n-1}} c(k,s_1,s_2,\ldots,s_{n-1}) = 0,
	\]
so $\lim_{s_1} \cdots \lim_{s_{n-1}} c(k,s_1,\ldots,s_{n-1}) = 0$.  This completes the proof of the claim.
	
Since $c$ is an instance of $\mathsf{D}^n_2$, fix $i < 2$ and an infinite limit-homogeneous set $L$ for $c$ with color $i$. By the claim, 
$p_k \in R_{1-i}$ for all $k \in L$. List the non-zero elements of $L$ as $k_0 < k_1 < \cdots$, and let $p \in (\omega)^\omega$ be the partition whose blocks are $[0,k_0)$ and  
$[k_m,k_{m+1})$ for $m \in \omega$. Each $x \in (p)^2$ satisfies $\mu^x(1) = k_m$ for some $m$. Since $R_0 \cup R_1$ is a reduced coloring, 
$x$ and $p_{k_m}$ have the same color, which is $R_{1-i}$. Since $x$ was arbitrary, $(p)^2 \subseteq R_{1-i}$ as required to complete this half of the theorem.
	
Next, we show that $\mathsf{D}^n_2 \leq_{\mathrm{sW}} \Delta^0_n$-$\mathsf{rDRT}^2_2$, and that $\mathsf{D}^n_2$ is implied 
by $\Delta^0_n$-$\mathsf{rDRT}^2_2$ over $\mathsf{RCA}_0$. (No extra induction is necessary for this implication.) Fix an instance $c : [\omega]^n \to 2$ of $\mathsf{D}^n_2$, 
and define a partition $R_0 \cup R_1$ of $(\omega)^2$ as follows. For $x \in (\omega)^2$ with $\mu^x(1) = k$, $x \in R_i$ for the unique $i$ such that
	\[
		\lim_{s_1} \cdots \lim_{s_{n-1}} c(k,s_1,\ldots,s_{n-1}) = i.
	\]
Since each of the iterated limits is assumed to exist over what follows on the right, we may express these limits by alternating $\Sigma^0_2$ and $\Pi^0_2$ definitions, as
	\[
		(\exists t_1 \forall s_1 \geq t_1) (\forall t_2 \geq s_1 \exists s_2 \geq t_2) \cdots c(k,s_1,\ldots,s_{n-1}) = i.
	\]
Thus, $R_0$ and $R_1$ are $\Sigma^0_n$-definable open subsets of $(\omega)^2$. By standard techniques, there are Borel codes for $R_0$ and $R_1$ as $\Sigma^0_n$ sets 
uniformly computable in $c$ and in $\RCA$. (Below, we illustrate this process for $\mathsf{D}^3_2$.)

By definition, $(\omega)^2 = R_0 \cup R_1$ is a reduced coloring and hence is an instance of $\Delta^0_n$-$\mathsf{rDRT}^2_2$. Let $p \in (\omega)^\omega$ be a solution to this 
instance, say with color $i < 2$.  Thus, for every $x \in (p)^2$, the limit color of $k = \mu^x(1)$ is $i$. Define $L = \{\mu^p(m) : m \geq 1\}$. Since for each $k \in L$, there is an 
$x \in (p)^2$ such that $\mu^x(1) = k$, $L$ is limit-homogeneous for $c$ with color $i$.

We end this proof by illustrating how to define the Borel codes for $R_0$ and $R_1$ as $\Sigma^0_3$ sets from a stable coloring $c(k,s_1,s_2)$. In this case, we have 
\[
\lim_{s_1} \lim_{s_2} c(k,s_1,s_2) = i \Leftrightarrow \exists t_1 (\forall s_1 \geq t_1 \, \forall t_2 \geq s_1) (\exists s_2 \geq t_2)~c(k,s_1,s_2) = i.
\]
The nodes in each $R_i$ are the initial segments of the strings $\langle \langle k,t_1 \rangle, \langle s_1,t_2 \rangle, s_2 \rangle$ for $k \leq t_1 < s_1 \leq t_2 < s_2$ and  
the labeling functions are $\varphi_i(\sigma) = \cup$ if $|\sigma| \in \{ 0,2 \}$, $\varphi_i(\sigma) = \cap$ if $|\sigma|=1$ and $\varphi_i(\langle \langle k,t_1 \rangle, \langle s_1,t_2 \rangle, s_2 \rangle) = 
[0^k1]$ if $c(k,s_1,s_2) = i$ and is equal to $\emptyset$ if $c(k,s_1,s_2) = 1-i$. It is straightforward to check in $\RCA$ that $R_i$ represents the union of clopen sets $[0^k1]$ such that 
the limit color of $k$ is $i$. 
\end{proof}

\section{Reverse math and Borel codes}
\label{sec:codes}

\subsection{Equivalence of the Borel and Baire versions over $\ATR$} 
\label{subsec:Borel_version}
In this subsection we show that over the base theory $\ATR$, the Baire
and Borel versions of the Dual Ramsey Theorem are equivalent.
We make the following definition in reverse mathematics.
\begin{defn}[$\RCA$]\label{defn:borel}
A \emph{Borel code} is a pair
$(B,\varphi)$, where $B \subseteq \omega^{<\omega}$ 
is well-founded and $\varphi$ is a labeling function 
as in Definition \ref{defn:effective_borel_codes}.
\end{defn}

This definition differs slightly from the definition of a Borel code 
which is found in the standard reference \cite{sosa}.  In that 
treatment, there is no labeling function, but certain conventions 
on the strings in $B$ determine the labels. Because there is no labeling
function, the set of leaves of $B$ may not be guaranteed to exist
in weak theories.  In \cite{sosa}, the base theory for anything to 
do with Borel sets is $\ATR$, so this distinction is never used.  
We would like to consider weaker 
base theories.  When the base theory is weaker,  
a constructive presentation of a Borel code should include
knowledge of which nodes are leaves.  For example, this 
leaf-knowledge was used in the proof of Theorem \ref{thm:sW}.
This is the reason for 
including the labeling function in our definition.  

In Section \ref{sec:weihrauch} we diverged from the standard
definition in a second way, by ascertaining membership
in a $\mathbf \Sigma^0_n$ set \emph{virtually}.  The standard 
method, which we use in this section, is via evaluation maps.
\begin{defn}[$\RCA$] \label{defn:eval}
Let $(B,\varphi)$ be a Borel code and $x \in (\omega)^k$.
An \emph{evaluation map} for $B$ at $x$ is a function 
$f:B \rightarrow \{ 0,1 \}$ such that
\begin{itemize}
\item For leaves $\sigma \in B$, $f(\sigma) = 1$ if and only if $x \in \varphi(\sigma)$.
\item If $\varphi(\sigma) = \cup$, $f(\sigma) = 1$ if and only if there 
exists $n$ such that $\sigma\concat n \in B$ and $f(\sigma\concat n) = 1$.
\item If $\varphi(\sigma) = \cap$, $f(\sigma) = 1$ if and only if for all $n$ such that $\sigma\concat n \in B$, $f(\sigma\concat n) = 1$.
\end{itemize}
 We say $x \in B$ if there is an 
evaluation map with value 1 at the root, and we say $x \not \in B$ if there is an evaluation map with value 0 at the root. 
\end{defn}

Observe that both $x \in B$ and $x \not \in B$ 
are $\Sigma^1_1$ statements. In general, $\ATR$ is required to show that evaluation maps exist. Similarly, $(\omega)^k = C_0 \cup \ldots \cup C_{\ell-1}$ is the 
$\Pi^1_2$ statement that for every $x \in (\omega)^k$ and $i < \ell$, there is an evaluation map for $C_i$ at $x$ and for some $i < \ell$, $x \in C_i$.

\begin{defn}[$\RCA$]\label{defn:bairecode}
Let $B$ be a Borel (or open or closed) code for subset of $(\omega)^k$. A \textit{Baire code for} $B$ consists of open sets $U$ and $V$ and a sequence 
$\langle D_n : n \in \omega \rangle$ of 
dense open sets such that $U \cup V$ is dense and for every $p \in \cap_{n \in \omega} D_n$, if $p \in U$ then $p \in B$ and if $p \in V$ then $p \not \in B$.  
\end{defn}

\begin{defn}[$\RCA$]
A \textit{Baire code} for a Borel coloring $(\omega)^k = C_0 \cup \cdots \cup C_{\ell-1}$ consists of open sets $O_i$, $i < \ell$, and a sequence 
$\langle D_n : n \in \omega \rangle$ of dense open sets such that $\cup_{i < \ell} O_i$ is dense and for every $p \in \cap_{n \in \omega} D_n$ and $i < \ell$, if 
$p \in O_i$ then $p \in C_i$.
\end{defn}

We confirm that $\ATR$ proves that every 
Borel set has the property of Baire.  This is just an effectivization of
the usual proof.

\begin{prop}[$\ATR$]
\label{prop:BP}
Every Borel code for a subset of $(\omega)^k$ has a Baire code. 
\end{prop}

\begin{proof}
Fix a Borel code $B$. For $\sigma \in B$, let 
$B_{\sigma} = \{ \tau \in B : \tau \text{ is comparable to } \sigma \}$. $B_{\sigma}$ is a Borel code for the set coded coded by the subtree of $B$ above $\sigma$ in the 
following sense. Let $f$ be an evaluation map for $B$ at $x$. The function $g: B_{\sigma} \rightarrow 2$ defined by $g(\tau) = f(\tau)$ for $\tau \succeq \sigma$ and 
$g(\tau) = f(\sigma)$ for $\tau \prec \sigma$ is an evaluation map for $B_{\sigma}$ at $x$ which witnesses $x \in B_{\sigma}$ if and only if $f(\sigma) = 1$. We 
denote this function $g$ by $f_{\sigma,x}$. 

Formally, our proof proceeds in two steps. First, by arithmetic transfinite recursion on the Kleene-Brouwer order $KB(B)$, we construct 
open sets $U_{\sigma}$, $V_{\sigma}$ and $D_{n,\sigma}$, $n \in \omega$, which are intended to form a Baire code for $B_{\sigma}$. This construction is essentially identical 
to the proof of Proposition \ref{prop:baire_degrees}. Second, for any $x \in (\omega)^k$ 
and evaluation map $f$ for $B$ at $x$, we show by arithmetic transfinite induction on $KB(B)$ that if $x \in \cap_{n \in \omega} D_{n,\sigma}$, then $x \in U_{\sigma}$ implies 
$x \in B_{\sigma}$ via $f_{\sigma,x}$ and $x \in V_{\sigma}$ implies $x \not \in B_{\sigma}$ via $f_{\sigma,x}$. For ease of presentation, we combine these two steps.  
Since $\ATR$ suffices to construct evaluation maps, we treat Borel codes as sets in a naive manner and suppress explicit mention of the evaluation maps.

If $\sigma$ is a leaf coding a basic clopen set $[\tau]$, we set $U_{\sigma} = [\tau]$, $V_{\sigma} = \overline{[\tau]}$ and $D_{n,\sigma} = (\omega)^k$. 
Similarly, if $\sigma$ codes $\overline{[\tau]}$, we switch the values of $U_{\sigma}$ and $V_{\sigma}$. In either case, it is clear that these open sets form a 
Baire code for $B_{\sigma}$. 

Suppose $\sigma$ is an internal node coding a union, so $B_{\sigma}$ is the union of $B_{\sigma^{\smallfrown}k}$ for $\sigma^{\smallfrown}k \in B$. We define 
$U_{\sigma}$ to be the union of $U_{\sigma^{\smallfrown}k}$ for $\sigma^{\smallfrown}k \in B$ and $V_{\sigma}$ to be the union of $[\tau]$ such that 
$V_{\sigma^{\smallfrown}k}$ is dense in $[\tau]$ for all $\sigma^{\smallfrown}k \in B$. The sequence $D_{n,\sigma}$ is the sequence of all open sets of the 
form $D_{n,\sigma^{\smallfrown}k} \cap (U_{\sigma^{\smallfrown}k} \cup V_{\sigma^{\smallfrown}k})$ for $n \in \omega$ and $\sigma^{\smallfrown}k \in B$. 
As in the proof of Proposition \ref{prop:baire_degrees}, $U_{\sigma} \cup V_{\sigma}$ and each $D_{n,\sigma}$ are dense. 

Let $x \in \cap_{n \in \omega} D_{n, \sigma}$. Suppose $x \in U_{\sigma}$ and we show $x \in B_{\sigma}$. By the definition of $U_{\sigma}$, fix 
$\sigma^{\smallfrown}k \in B$ such that $x \in U_{\sigma^{\smallfrown}k}$. Since $x \in \cap_{n \in \omega} D_{n, \sigma^{\smallfrown}k}$, we have by induction that 
$x \in B_{\sigma^{\smallfrown}k}$ and hence $x \in B_{\sigma}$. On the other hand, suppose $x \in V_{\sigma}$ and we show $x \not \in B_{\sigma}$. Fix $\tau$ such that 
$\tau \prec x$ and $[\tau] \subseteq V_{\sigma}$, and fix $k$ such that 
$\sigma^{\smallfrown}k \in B$. Since $x \in \cap_{n \in \omega} D_{n,\sigma}$, $x \in U_{\sigma^{\smallfrown}k} \cup V_{\sigma^{\smallfrown}k}$. However, 
$V_{\sigma^{\smallfrown}k}$ is dense in $[\tau]$. Therefore, $x \not \in U_{\sigma^{\smallfrown}k}$ (because $U_{\sigma^{\smallfrown}k}$ and $V_{\sigma^{\smallfrown}k}$ 
must be disjoint as in the proof of Proposition \ref{prop:baire_degrees}), 
so $x \in V_{\sigma^{\smallfrown}k}$. Since $x \in \cap_{n \in \omega} D_{n,\sigma^{\smallfrown}k}$, we have by induction that 
$x \not \in B_{\sigma^{\smallfrown}k}$. Because this holds for every $\sigma^{\smallfrown}k \in B$, it follows that $x \not \in B_{\sigma}$, completing the case for unions.

The case for an interior node coding an intersection is similar with the roles of $U_{\sigma}$ and $V_{\sigma}$ switched.\end{proof}

\begin{prop}[$\ATR$]
$\BaDRT^k_{\ell}$ implies $\BoDRT^k_{\ell}$.
\end{prop}

\begin{proof}
By Proposition \ref{prop:BP},
fix Baire codes $U_i$, $V_i$ and $D_{n,i}$ for each $C_i$. We claim that the open sets $U_i$ for $i < \ell$ and 
the sequence of dense open sets $D_{n,i}$ for $i < \ell$ and $n < \omega$ form a Baire code for this coloring. Note that if $i < \ell$ and $x \in \cap_{n,i} D_{n,i}$, 
then $x \in U_i$ implies $x \in C_i$. Therefore, it suffices to show that $\cup_{i < \ell} U_i$ is dense.

Suppose not.  Then there is $\tau$ such that $[\tau] \cap U_i = \emptyset$ for all $i$.  Because each set $U_i \cup V_i$ is open and dense, by the 
Baire Category Theorem there is $x \in [\tau]$ such that $x \in \cap_{n\in \omega, i<\ell} D_{n,i}$ and 
$x \in \cap_{i < \ell} (U_i \cup V_i)$.  Since $x$ is not in any $U_i$, we have $x \in V_i$ for each $i$. Therefore, for each $i$, $x \not\in C_i$, 
contradicting that $(\omega)^k = C_0 \cup \dots \cup C_{\ell-1}$.
\end{proof}

\begin{lem}[$\RCA$]
For every code $O$ for an open set, there is a Borel code $B$ such that $(\omega)^k = B \cup \overline{B}$ and for all $x \in (\omega)^k$, $x \in B$ if and only if 
$x \in O$. 
\end{lem}

\begin{proof}
The content here lies in the proof that $(\omega)^k = B \cup \overline{B}$.  That is,
we need to show that in the obvious Borel code, 
every $x \in (\omega)^k$ has an evaluation map.

Fix $O$. Let $(B,\varphi)$ be the Borel code consisting of a root and a single leaf
for each $\langle s, \tau\rangle \in O$, where the leaf is labeled with $[\tau]$.

 We claim that for every $x \in (\omega)^k$, there is a unique evaluation map 
 $f$ for $B$ at $x$, and 
$f(\lambda) = 1$ if and only if $x \in O$. To prove this claim, we define two 
potential evaluation maps, $f_0$ and $f_1$. 
Let $f_0(\lambda) = 0$ and $f_1(\lambda)=1$.  Then for each $i\in \{0,1\}$ 
and each leaf $\sigma$ with label $\tau$, define $f_i(\sigma) = 1$ if 
and only if $x \in [\tau]$. Both these functions have $\Delta^0_1(x,B,\varphi)$
definitions, and exactly one of them satisfies the condition to be 
an evaluation map.  Clearly, this condition implies that $x \in B$ 
if and only if $x \in O$.
\end{proof}

\begin{cor}[$\RCA$]
$\BoDRT^k_{\ell}$ implies $\BaDRT^k_{\ell}$. 
\end{cor}
\begin{proof} The previous proposition shows that
$\BoDRT^k_{\ell}$ implies $\ODRT^k_{\ell}$ and hence implies $\BaDRT^k_{\ell}$. 
\end{proof}

\subsection{The strength of ``Every Borel set has the property of Baire''}

We have just seen that over $\ATR$, the Borel and Baire versions
of the Dual Ramsey Theorem are equivalent.  But only one 
direction used $\ATR$, in order to assert that every Borel set 
has the property of Baire.  In this section, we ask if this principle 
really requires $\ATR$.  We find that it does, but the reason
is unsatisfactory, because it depends
on a technicality in the standard definition of a Borel
set.  Some of the authors of the present paper removed that 
technicality in the later-researched but earlier-appearing paper 
\cite{ADMSW}.  When the technicality is removed, a principle 
strictly weaker than $\ATR$ emerges.  
We refer the reader to \cite{ADMSW}
for details.

In this section we show:
\begin{thm}[$\RCA$]
\label{ATRequiv}
The following are equivalent. 
\begin{enumerate}
\item $\ATR$.
\item For every Borel code $B$ for a subset of $(\omega)^k$, there is an $x \in (\omega)^k$ such that $x \in B$ or $x \not \in B$. 
\item Every Borel code $B$ for a subset of $(\omega)^k$ has a Baire code. 
\end{enumerate}
\end{thm}

In fact, the implication from (2) to (1) can be witnessed using only \emph{trivial}
Borel codes, which we define as follows.
\begin{defn}[$\RCA$]
A Borel code $(B,\varphi)$ for a subset of $(\omega)^k$
is \emph{trivial} if every leaf is labeled
with either $\emptyset$ or $(\omega)^k$.
\end{defn}
If $B$ is a trivial Borel code, then an evaluation map for $B$ at $p$ is independent of $p$, so we can refer to an evaluation map $f$ for $B$. 
Because we work with 
trivial Borel codes, the underlying topological space does not matter as long as Borel codes are defined in a manner similar to 
Definitions \ref{defn:borel} and \ref{defn:eval}. For example, Theorem \ref{ATRequiv} holds for Borel codes of subsets of $2^{\omega}$ or $\omega^{\omega}$ as 
defined in Simpson \cite{sosa}.  (The fact that the leaves are labeled in Definition \ref{defn:borel}
does not affect any of the arguments in this section.)

The main ideas in the proof that (2) implies (1) use effective transfinite 
recursion and are similar to those in Section 7.7 of Ash and Knight \cite{AK}.

\begin{prop}[$\RCA$]
The statement ``every trivial Borel code has an evaluation map'' implies $\ACA$. 
\end{prop}

\begin{proof}
Fix $g: \omega \rightarrow \omega$ and we show $\text{range}(g)$ exists. Let $B$ be the trivial Borel code consisting of the initial segments of 
$\langle n,m,1 \rangle$ for $g(m) = n$ and $\langle n,m,0 \rangle$ for $g(m) \neq n$. 
Label all leaves which end in 0 with $\emptyset$, and label all leaves
which end in 1 with the entire space.  Label all interior nodes with $\cup$.
Let $f$ be an evaluation map for $B$.  Then $f(\langle n\rangle) = 1$
if and only if there is an $m$ such that $g(m)=n$.
\end{proof}

In order to strengthen this result to imply $\ATR$, we need to verify that
effective transfinite recursion works in $\ACA$.
Let $LO(X)$ and $WO(X)$ be the standard formulas in second order arithmetic saying $X$ is a linear order and $X$ is a well order. We abuse notation and write $x \in X$ in  
place of $x \in \text{field}(X)$. For a formula $\varphi(n,X)$, $H_{\varphi}(X,Y)$ is the formula stating $LO(X)$ and $Y = \{ \langle n,j \rangle : j \in X  \wedge  
\varphi(n,Y^j) \}$ where $Y^j = \{ \langle m,a \rangle : a <_X j \wedge \langle m,a \rangle \in Y \}$. When $\varphi$ is arithmetic, $H_{\varphi}(X,Y)$ is 
arithmetic and $\ACA$ proves that if $WO(X)$, then there is at most one $Y$ such that $H_{\varphi}(X,Y)$. We define our formal version of effective transfinite recursion. 

\begin{defn}
$\mathsf{ETR}$ is the axiom scheme
\[
\forall X \, \Big[ \big( WO(X) \wedge \forall Y \, \forall n \, (\varphi(n,Y) \leftrightarrow \neg \psi(n,Y)) \big) \rightarrow \exists Y \, H_{\varphi}(X,Y) \Big]
\]
where $\varphi$ and $\psi$ range over $\Sigma^0_1$ formulas. 
\end{defn}

We show that $\mathsf{ETR}$ is provable in $\ACA$. Following Simpson \cite{sosa}, we avoid using the recursion theorem and note that 
the only place the proof goes beyond $\RCA$ is in the use of transfinite induction for $\Pi^0_2$ formulas, which holds is $\ACA$ and is equivalent to transfinite induction for 
$\Sigma^0_1$ formulas. Greenberg and Montalb\'{a}n \cite{GM} point out that $\mathsf{ETR}$ can also be proved using the recursion theorem, although this proof also uses 
$\Sigma^0_1$ transfinite induction. 

\begin{prop}
$\mathsf{ETR}$ is provable in $\ACA$. 
\end{prop}

\begin{proof}
Fix a well order $X$ and $\Sigma^0_1$ formulas $\varphi$ and $\psi$. Throughout this proof, we let $f$, $g$ and $h$ be variables 
denoting finite partial functions from $\omega$ to $\{ 0,1 \}$ coded in the canonical way as finite sets of ordered pairs. We write $f \preceq g$ 
(or $f \prec X$) if $f \subseteq g$ (or $f \subseteq \chi_X$) as sets of ordered pairs. By the usual normal form results (e.g.~Theorem II.2.7 in Simpson), 
we fix a $\Sigma^0_0$ formula $\varphi_0$ such that
\[
\forall Y \, \forall n \, \big( \varphi(n,Y) \leftrightarrow \exists f \, (f \prec Y \wedge \varphi_0(n,f)) \big)
\]
and such that if $\varphi_0(n,f)$ and $f \prec g$, then $\varphi_0(n,g)$. We fix a formula $\psi_0$ related to $\psi$ in the same manner. Since 
$\varphi(n,Y) \leftrightarrow \neg \psi(n,Y)$, we cannot have compatible $f$ and $g$ such that $\varphi_0(n,f)$ and $\psi_0(n,g)$. 

Our goal is to use partial functions $f$ as approximations to a set $Y$ such that $H_{\varphi}(X,Y)$. Therefore, we view $\text{dom}(f)$ as consisting of coded pairs 
$\langle n,a \rangle$. For $f$ to be a suitable approximation to $Y$, we need that if $\langle n,a \rangle \in \text{dom}(f)$ and $a \not \in X$, then $f(\langle n,a \rangle) = 0$. 
Similarly, if $f$ is an approximation to $Y^j$, we need that $f(\langle n,a \rangle) = 0$ whenever $\langle n,a \rangle \in \text{dom}(f)$ and $a \geq_X j$. These observations 
motivate the following definitions.

Let $f$ be a finite partial function and let $i \in X$. We define 
\[
f^i = f \upharpoonright \{ \langle n,a \rangle : n \in \omega \wedge a <_X i \}.
\]
We say $g \succeq f$ is an $i$-\textit{extension of} $f$ if for all $\langle n,a \rangle \in \text{dom}(g) - \text{dom}(f)$, $g(\langle n,a \rangle) = 0$ and either 
$a \not \in X$ or $i \leq_X a$. 

For $j \in X$, $f$ is a $j$-\textit{approximation} if the following conditions hold.
\begin{itemize}
\item If $\langle n,a \rangle \in \text{dom}(f)$ with $a \not \in X$ or $j \leq_X a$, then $f(\langle n,a \rangle) = 0$.
\item If $\langle n,a \rangle \in \text{dom}(f)$ and $a <_X j$, then 
\begin{itemize}
\item if $f(\langle n,a \rangle) = 1$, then there is an $a$-extension $h$ of $f^a$ such that $\varphi_0(n,h)$, and
\item if $f(\langle n,a \rangle) = 0$, then there is an $a$-extension $h$ of $f^a$ such that $\psi_0(n,h)$.
\end{itemize}
\end{itemize}
Note that if $f$ is a $j$-approximation and $i <_X j$, then $f^i$ is an $i$-approximation. Also, if $f$ is a $j$-approximation and $g$ is a $j$-extension of $f$, then $g$ 
is a $j$-approximation.

\begin{claim}
For all $j \in X$, there do not exist $m \in \omega$ and $j$-approximations $f$ and $g$ such that $\varphi_0(m,f)$ and $\psi_0(m,g)$.  
\end{claim}

The proof is by transfinite induction on $j$. Fix the least $j \in X$ for which this property fails and fix witnesses $m$, $f$ and $g$. To derive a contradiction, 
it suffices to show that $f$ and $g$ are compatible. Fix $\langle k,a \rangle$ such that both $f(\langle k,a \rangle)$ and $g(\langle k,a \rangle)$ are defined. If $a \not \in X$ 
or $j \leq_X a$, then $f(\langle k,a \rangle) = g(\langle k,a \rangle) = 0$. 

Suppose for a contradiction that $a <_X j$ and $f(\langle k,a \rangle) \neq g(\langle k,a \rangle)$. Without loss of generality, $f(\langle k,a \rangle) = 1$ and 
$g(\langle k,a \rangle) = 0$. Fix $a$-extensions $h$ and $h'$ of $f^a$ and $g^a$ respectively such that $\varphi_0(k,h)$ and $\psi_0(k,h')$. 
Since $f$ is a $j$-approximation, $f^a$ is an $a$-approximation, and since $h$ is an $a$-extension of $f^a$, $h$ is also an $a$-approximation. Similarly, $h'$ is an 
$a$-approximation. Therefore, we have $k \in \omega$, $a <_X j$ and $a$-approximation $h$ and $h'$ such that $\varphi_0(k,h)$ and $\psi_0(k,h')$ contradicting the 
minimality of $j$.

\begin{claim}
For any $j$-approximation $f$ and any $m \in \omega$, there is a $j$-approximation $g \succeq f$ such that either $\varphi_0(m,g)$ or $\psi_0(m,g)$. 
\end{claim}

The proof is again by transfinite induction on $j$. Fix the least $j$ for which this property fails and fix witnesses $f$ and $m$. Let $\langle n_s,i_s \rangle$ enumerate 
the pairs not in the domain of $f$. Below, we define a sequence $f=f_0 \preceq f_1 \preceq \cdots$ of $j$-approximations such that $f_{s+1}(\langle n_s,i_s \rangle)$ is 
defined. Let $Y$ be the set with $\chi_Y = \cup_s f_s$. Either $\varphi(m,Y)$ or $\psi(m,Y)$ holds, and so 
there is a $g \prec Y$ such that $\varphi_0(m,g)$ or $\psi_0(m,g)$ holds. Fixing $s$ such that $g \preceq f_s$ shows that either $\varphi_0(m,f_s)$ or $\psi_0(m,f_s)$ 
holds for the desired contradiction. 

To define $f_{s+1}$, we need to extend $f_s$ to a $j$-approximation $f_{s+1}$ with $\langle n_s,i_s \rangle \in \text{dom}(f_{s+1})$. We break into several cases. 
If $f_s(\langle n_s,i_s \rangle)$ is already defined, let $f_{s+1} = f_s$. Otherwise, if $i_s \not \in X$ or $j \leq_X i_s$, set $f_{s+1}(\langle n_s,i_s \rangle) = 0$ and leave 
the remaining values as in $f_s$. In both cases, it is clear that $f_{s+1}$ is a $j$-approximation.

Finally, if $i_s <_X j$ and $f_s(\langle n_s,i_s \rangle)$ is undefined, we apply the induction hypothesis to the $i_s$-approximation $f_s^{i_s}$ to get an $i_s$-approximation 
$g \succeq f_s^{i_s}$ such that either $\varphi_0(n_s,g)$ holds or $\psi_0(n_s,g)$ holds. Define $f_{s+1}$ as follows. 
\begin{itemize}
\item For $\langle m,a \rangle \in \text{dom}(g)$ with $a <_X i_s$, set $f_{s+1}(\langle m,a \rangle) = g(\langle m,a \rangle)$. 
\item For $\langle m,a \rangle \in \text{dom}(f_s)$ with $i_s \leq_X a$ or $a \not \in X$, set $f_{s+1}(\langle m,a \rangle) = f_s(\langle m,a \rangle)$.
\item Set $f_{s+1}(\langle n_s,i_s \rangle) = 1$ if $\varphi_0(n_s,g)$ holds and $f_{s+1}(\langle n_s,i_s \rangle) = 0$ if $\psi_0(n_s,g)$ holds. 
\end{itemize}
It is straightforward to verify that $f_s \prec f_{s+1}$, $g$ is an $i_s$-extension of $f_{s+1}^{i_s}$ and $f_{s+1}$ is a $j$-approximation, completing the proof of the claim.

We define the set $Y$ for which we will show $H_{\varphi}(X,Y)$ holds by $\langle m,j \rangle \in Y$ if and only if $j \in X$ and there is a $j$-approximation $f$ such that $\varphi_0(m,f)$. It follows from the claims above that 
$\langle m,j \rangle \not \in Y$ if and only if either $j \not \in X$ or there is a $j$-approximation $f$ such that $\psi_0(m,f)$. Therefore, $Y$ has a $\Delta^0_1$ definition. 
The next two claims show that $H_{\varphi}(X,Y)$ holds, completing our proof.

\begin{claim}
If $f$ is a $j$-approximation, then $f \prec Y^j$. 
\end{claim}

Consider $\langle m,a \rangle \in \text{dom}(f)$. If $a \not \in X$ or $j \leq_X a$, then $f(\langle m,a \rangle) = Y^j(\langle m,a \rangle) = 0$. Suppose $a <_X j$. If 
$f(\langle m,a \rangle) = 1$, then there is an $a$-extension $g$ of $f^a$ such that $\varphi_0(m,g)$. Since $f^a$ is an $a$-approximation and $g$ is an $a$-extension of 
$f^a$, $g$ is an $a$-approximation. Therefore, $\langle m,a \rangle \in Y$ by definition and hence $\langle m,a \rangle \in Y^j$. By similar reasoning, 
if $f(\langle m,a \rangle) = 0$, then $\langle m,a \rangle \not \in Y$ and hence $\langle m,a \rangle \not \in Y^j$. 

\begin{claim}
$\langle m,j \rangle \in Y$ if and only if $\varphi(m,Y^j)$. 
\end{claim}

Assume that $\langle m,j \rangle \in Y$ and fix a $j$-approximation $f$ such that $\varphi_0(m,f)$. Since $f \prec Y^j$, $\varphi(m,Y^j)$. 
For the other direction, assume that $\varphi(m,Y^j)$. Fix a $j$-approximation $f$ such that either $\varphi_0(m,f)$ or $\psi_0(m,f)$. Since $f \prec Y^j$ and 
$\varphi(m,Y^j)$, we must have $\varphi_0(m,f)$ and therefore $\langle m,j \rangle \in Y$ by definition. 
\end{proof}

We recall some notation and facts from Simpson \cite{sosa} to state the equivalence of $\mathsf{ATR}_0$ we will prove. We let $TJ(X)$ denote the Turing jump in $\ACA$ given 
by fixing a universal $\Pi^0_1$ formula. We use the standard recursion theoretic notations $\Phi_e^X$ and $\Phi_{e,s}^X$ with the understanding that 
they are defined by this fixed universal formula. 

$\mathcal{O}_+(a,X)$ is the arithmetic statement that $a = \langle e,i \rangle$, $e$ is an $X$-recursive index of an $X$-recursive linear order $\leq_e^X$ 
and $i \in \text{field}(\leq_e^X)$. $\mathcal{O}_+^X = \{ a : \mathcal{O}_+(a,X) \}$ exists in $\ACA$. For $a, b \in \mathcal{O}_+^X$, we write $b <_{\mathcal{O}}^X a$ 
if $a = \langle e,i \rangle$, $b = \langle e,j \rangle$ and $j <_e^X i$. For $a \in \mathcal{O}_+^X$, the set $\{ b : b <_{\mathcal{O}}^X a \}$ exists in $\ACA$. 

$\mathcal{O}(a,X)$ is the $\Pi^1_1$ statement 
$\mathcal{O}_+(a,X) \wedge WO(\{ b : b <_{\mathcal{O}}^X a \})$. Intuitively, $\mathcal{O}(a,X)$ says that $a = \langle e,i \rangle$ is an $X$-recursive ordinal notation 
for the well ordering given by the restriction of $\leq_e^X$ to $\{ j : j <_e^X i \}$. In $\mathsf{ATR}_0$, if $\mathcal{O}(a,X)$, then the set 
\[
H_a^X = \{ \langle y,0 \rangle : y \in X \} \cup \{ \langle y, b+1 \rangle : b <_{\mathcal{O}}^X a \wedge y \in TJ(H_b^X) \}
\]
exists. In $\ACA$, there is an arithmetic formula $H(a,X,Y)$ which, under the assumption that $\mathcal{O}(a,X)$, holds if and only 
if $Y = H_a^X$.  

By Theorem VIII.3.15 in Simpson \cite{sosa}, $\mathsf{ATR}_0$ is equivalent over $\ACA$ to 
\[
\forall X \, \forall a \, ( \mathcal{O}(a,X) \rightarrow H_a^X \text{ exists}).
\]
If $\mathcal{O}(a,X)$ with $a = \langle e,i \rangle$, then we can assume without loss of generality that there are 
$a'$ and $a''$ such that $\mathcal{O}(a',X)$, $\mathcal{O}(a'',X)$ and $a <_{\mathcal{O}}^X a' <_{\mathcal{O}}^X a''$ by adding two new successors 
of $i$ in $\leq_e^X$ if necessary. Therefore, to prove $\mathsf{ATR}_0$, it suffices to fix $a$ and $X$ such that $\mathcal{O}(a,X)$ and prove 
$\forall c <_{\mathcal{O}}^X b \, (H_c^X \text{ exists})$ for each $b <_{\mathcal{O}}^X a$.  

\begin{thm}[$\ACA$]
\label{thm:trivial}
The statement ``every trivial Borel code has an evaluation map'' implies $\mathsf{ATR}_0$. 
\end{thm}

\begin{proof}
Fix $a$ and $X$ such that $\mathcal{O}(a,X)$, so the restriction of $<_{\mathcal{O}}^X$ to $\{ b : b <_{\mathcal{O}}^X a \}$ is a well order. 
Using $\mathsf{ETR}$, we define trivial Borel codes $B_{x,b}$ for $x \in \omega$ by transfinite recursion on $b <_{\mathcal{O}}^X a$. We explain the intuitive construction 
before the formal definition. 

Let $b <_{\mathcal{O}}^X a$ and $x \in \omega$. We want to define a trivial Borel code $B_{x,b}$ such that if $f$ is an evaluation map for $B_{x,b}$, then 
$f(\lambda) = 1$ if and only if 
$x \in TJ(H_b^X)$. We label $\lambda$ with $\cup$. For each binary string 
$\sigma$ such that $\Phi_{x,|\sigma|}^{\sigma}(x)$ converges, we add a successor $\langle n_\sigma \rangle$.  Here $\sigma\mapsto n_\sigma$ is just some 
primitive recursive bijection between $2^{<\omega}$ and $\omega$.
It follows that $f(\lambda) =1$ if and only if there is a $\sigma$ such that $\Phi_{x,|\sigma|}^{\sigma}(x)$ converges and 
$f(\langle n_\sigma \rangle) = 1$. (In case $\Phi_{x,|\sigma|}^{\sigma}(x)$ always diverges, 
we may also add a leaf $\langle n \rangle$ which is labeled with $\emptyset$. In this case, $f(\lambda) = f(\langle n \rangle) = 0$ and $x \not \in TJ(H_b^X)$ which is what we want.)

Next, we want to ensure $f(\langle n_\sigma \rangle) = 1$ if and only if $\sigma \prec H_b^X$.  We label $\langle n_\sigma\rangle$ with $\cap$,
and for each $k < |\sigma|$, we add a 
successor $\langle n_\sigma, k \rangle$. 
We want $f(\langle n_\sigma, k \rangle) = 1$ if and only if $\sigma(k) = H_b^X(k)$. We 
break into cases to determine the extensions of $\langle n_\sigma, k \rangle$. 

For the first case, suppose $k = \langle y,0 \rangle$. We want $f(\langle n_\sigma, k \rangle) = 1$ if and only if $y \in X$. If $\sigma(k) = X(y)$, we label this node with 
the entire space, and if $\sigma(k) \neq X(y)$, we label this node with $\emptyset$. In either case, the successor nodes will be leaves so we have $f(\langle n_\sigma, k \rangle) = 1$ if and only 
if $k \in H_b^X$. 

For the second case, suppose $k = \langle y, c+1 \rangle$ and $c <_{\mathcal{O}}^X b$. By the induction hypothesis, we have defined the trivial Borel code $B_{y,c}$ already. If 
$\sigma(k) = 1$, then we label $\langle n_\sigma, k \rangle$
with $\cup$, and attach to it a copy of $B_{y,c}$, treating $\langle n_\sigma, k \rangle$ as the root of $B_{y,c}$. 
The map $f$ restricted to the subtree above $\langle n_\sigma, k \rangle$ is an evaluation map for $B_{y,c}$ and hence by the inductive hypothesis
\[
f(\langle n_\sigma, k \rangle) = 1 \Leftrightarrow y \in TJ(H_c^X) \Leftrightarrow k \in H_b^X \Leftrightarrow \sigma(k) = H_b^X(k).
\] 
On the other hand, if $\sigma(k) = 0$, then we label $\langle n_\sigma, k \rangle$
with $\cap$ and extend it by a copy of $\overline{B}_{y,c}$. By the inductive hypothesis, we have 
\[
f(\langle n_\sigma, k \rangle) = 1 \Leftrightarrow y \not \in TJ(H_c^X) \Leftrightarrow k \not \in H_b^X \Leftrightarrow \sigma(k) = H_b^X(k).
\]

For the third case, suppose that $k = \langle y,c+1 \rangle$ and $c \not <_{\mathcal{O}}^X b$. In this case, we know $H_b^X(k) = 0$. If $\sigma(k) = 0$,
we label $\langle n_\sigma, k\rangle$ with the entire space, and if $\sigma(k) =1$
we label it with $\emptyset$.  

The formal construction follows this outline. To simplify the notation, for a trivial Borel code $B$, we let $B^1 = B$ and $B^0 = \overline{B}$. 
Since ``$\Phi_{x,|\sigma|}^{\sigma}(x)$ converges'' is a bounded quantifier statement and 
$c <_{\mathcal{O}}^X b$ is a $\Delta^0_1$ statement with parameter $X$, the following recursion on $b <_{\mathcal{O}}^X a$ can be done with $\mathsf{ETR}$. 
For each $x \in \omega$, we put $\lambda$ in $B_{x,b}$ with label $\cup$.
For each $\sigma$ such that $\Phi_{x, |\sigma|}^{\sigma}(x)$ converges, we put $\langle n_\sigma \rangle$ and $\langle n_\sigma, k \rangle$ in $B_{x,b}$ for all $k < |\sigma|$. 
We label $\langle n_\sigma\rangle$ with $\cap$.
We extend $\langle n_\sigma, k \rangle$ as follows.
\begin{itemize}
\item For $k = \langle y,0 \rangle$: if $\sigma(k) = X(y)$, then $\langle n_\sigma, k \rangle$
is labeled with the whole space, and if $\sigma(k) \neq X(y)$, then it is labeled
with $\emptyset$.
\item For $k=\langle y,c+1 \rangle$ with $c <_{\mathcal{O}}^X b$, $\langle n_\sigma, k \rangle^{\smallfrown}\tau \in B_{x,b}$ for all $\tau \in B_{y,c}^{\sigma(k)}$, with 
labels inherited from $B_{y,c}^{\sigma(k)}$.
\item For $k=\langle y,c+1 \rangle$ with $c \not <_{\mathcal{O}}^X b$, $\langle n_\sigma, k\rangle$ gets labeled with the whole set if $\sigma(k)=0$ and labeled with
$\emptyset$ if $\sigma(k)=1$.
\end{itemize}

This completes the construction of the trivial Borel codes $B_{x,b}$ for $b <_{\mathcal{O}}^X a$ by $\mathsf{ETR}$. To complete the proof, we fix an arbitrary 
$b <_{\mathcal{O}}^X a$ and show that $\forall c <_{\mathcal{O}}^X b \, (H_c^X \text{ exists})$. 

Fix an index $x$ and $s \in \omega$ such that $\Phi_{x,s}^{1^s}(x)$ converges. Let $N$ be the least 
value of $s$ witnessing this convergence so $\Phi_{x,s}^{1^s}(x)$ converges for all $s \geq N$. Let $f$ be an evaluation map for $B_{x,b}$. 

For $c <_{\mathcal{O}}^X b$ and $y \in \omega$, let $\sigma = 1^{N+k}$ where $k = \langle y,c+1 \rangle$. Define 
$f_{y,c}(\tau) = f(\langle n_\sigma,k \rangle^{\smallfrown}\tau)$. We claim 
$f_{y,c}$ is an evaluation map for $B_{y,c}$. By the choice of $x$, $\Phi_{x,|\sigma|}^{\sigma}(x)$ converges.
Since $c <_{\mathcal{O}}^X b$ and $\sigma(k) = 1$, we have $\langle n_\sigma, k \rangle^{\smallfrown}\tau \in B_{x,b}$ if 
and only if $\tau \in B_{y,c}$. Therefore, $f_{y,c}$ is defined on $B_{y,c}$ and it satisfies the conditions for an evaluation map because $f$ does.

Recall that $H(x,X,Y)$ is a fixed arithmetic formula such that if $\mathcal{O}(x,X)$, then $H(x,X,Y)$ holds if and only if $Y= H_x^X$.  Define
\[
Z = \{ \langle y,0 \rangle : y \in X \} \cup \{ k : k=\langle y,c+1 \rangle \, \wedge \, c <_{\mathcal{O}}^X b \, \wedge \, f(\langle n_\sigma,k \rangle) = 1 \}.
\]
For $c <_{\mathcal{O}}^X b$, let $Z^c = \{ \langle y,r \rangle \in Z : r=0 \, \vee \, r-1 <_{\mathcal{O}}^X c \}$. 
We show the following properties by simultaneous arithmetic induction on $c <_{\mathcal{O}}^X b$.
\begin{enumerate}
\item $H(c,X,Z^c)$ holds. That is, $Z^c = H_c^X$. 
\item For all $y$, $f_{y,c}(\lambda) = 1$ if and only if $y \in TJ(Z^c) = TJ(H_c^X)$.
\end{enumerate} 
These properties imply $\forall c <_{\mathcal{O}}^X b \, (H_c^X \text{ exists})$ completing our proof.

Fix $c <_{\mathcal{O}}^X b$ and assume (1) and (2) hold for $d <_{\mathcal{O}}^X c$. To see (1) holds for $c$, fix $k$. If $k = \langle y,0 \rangle$, then 
$k \in Z^c \Leftrightarrow y \in X \Leftrightarrow k \in H_c^X$. Suppose $k = \langle y,d+1 \rangle$. If $d \not <_{\mathcal{O}}^X c$, then $k \not \in H_c^X$ and 
$k \not \in Z^c$. If $d <_{\mathcal{O}}^X c$, then
\[
k \in Z^c \Leftrightarrow f(\langle n_\sigma,k \rangle) = 1 \Leftrightarrow f_{y,d}(\lambda)=1.
\]
By the induction hypothesis, $k \in Z^c$ if and only if $y \in TJ(Z^d) = TJ(H_d^X)$, which holds if and only if $k \in H_c^X$, completing the proof of (1).

To prove (2), fix $y$ and let $k = \langle y, c+1 \rangle$. By definition, 
\[
k \in Z^c \Leftrightarrow f_{y,c}(\lambda) = f(\langle n_\sigma, k \rangle) = 1, 
\]
and $y \in TJ(Z^c) = TJ(H_c^X)$ if and only if there is a $\sigma$ such that $\Phi_{y,|\sigma|}^{\sigma}(y)$ converges and 
$\sigma \prec Z^c = H_c^X$. 

Suppose there are no $\sigma$ such that $\Phi_{y,|\sigma|}^{\sigma}(y)$ converges. In this case, $y \not \in TJ(H_c^X)$ and $f_{y,c}(\lambda) = 0$. Therefore $f_{y,c}(\lambda) = 1$ if and only if $y \in TJ(H_c^X)$ as required.

Suppose $\Phi_{y,|\sigma|}^{\sigma}(y)$ converges for some $\sigma$. For any such $\sigma$, $\langle n_\sigma, k \rangle \in B_{y,c}$ for all $k < |\sigma|$. 
By the induction hypothesis 
and the case analysis in the intuitive explanation of the construction, we have $f_{y,c}(\langle n_\sigma \rangle) = 1$ if and only if 
$\sigma \prec H_c^X = Z^c$, and therefore, $f_{y,c}(\lambda) = 1$ if and only if there is a $\sigma$ such that $\Phi_{y,|\sigma|}^{\sigma}(y)$ converges and 
$\sigma \prec H_c^X$, completing the proof of (2) and of the theorem. 
\end{proof}

We conclude with a proof of Theorem \ref{ATRequiv}. 

\begin{proof}
Lemma V.3.3 in Simpson \cite{sosa} shows (1) implies (2) in the space $2^{\omega}$ and the proof translates immediately to $(\omega)^k$. By Proposition \ref{prop:BP}, 
(1) implies (3). It follows from Theorem \ref{thm:trivial} that (2) implies (1). We show (3) implies (2). Let $B$ be a Borel code. Fix a Baire code  
$U$, $V$ and $D_n$ for $B$. Since each $D_n$ and $U \cup V$ is a dense open set, there is an $x \in (U \cup V) \cap \cap_{n \in \omega} D_n$. 
If $x \in U$, then by the definition of a Baire code, $x \in B$, and similarly, if $x \in V$, then $x \not \in B$. Therefore, we have a partition $x$ such that $x \in B$ or 
$x \not \in B$ as required. 
\end{proof}

\section{Open Questions}\label{sec:open}

While Figure \ref{fig:1.1} summarizes the known implications 
among the studied principles, in most cases it is not known 
whether the results are optimal.  It is particularly 
dissatisfying that the best upper bound for these principles 
remains $\Pi^1_1$-$\mathsf{CA}_0$.  Observe 
that, on the basis of the proof of $\CDRT^k_\ell$ given 
in Theorem \ref{thm.3.31}, any upper bound on the strength 
of the Carlson-Simpson Lemma $\CSL({k-1},\ell)$ would also
imply a related upper bound on the strength of $\CDRT^k_\ell$.
Therefore, it would be interesting to know the following:

\begin{question}
For any $k\geq 3$, does $\CSL(k,\ell)$ follow from $\ATR$?
\end{question}

The best known upper bound for $\CSL(2,\ell)$ is $\ACA$; it is 
shown in \cite{LiuMoninPatey} that the stronger principle 
$\OVW(2,\ell)$ follows from $\ACA$.  

Turning attention now to lower bounds, the principles 
$\CDRT^k_\ell$ for $k\geq 4$
 are not obviously implied by $\HT$ or 
$\ACA^+$.  We wonder whether an implication might go the 
other way.

\begin{question}
For any $k\geq 4$, does $\cDRT^k_\ell$ imply $\HT$ or $\ACA^+$?
\end{question}

When $k\geq 4$, it is known that $\CDRT^k_\ell$ implies $\ACA$ 
(this was proved for $\ODRT^k_\ell$ in \cite{ms}).
On the other hand, while $\cDRT^3_2$ is provable from 
Hindman's Theorem, the best lower bound we have 
on $\CDRT^3_2$ is $\RT^2_2$.  Furthermore, nothing about 
the relationship of $\CDRT^3_2$ and $\ACA$ is known.

\begin{question}
Is $\cDRT^3_2$ comparable to $\ACA$?
\end{question}

For the $k=2$ case, can Theorem \ref{thm:sW} be strengthened in the following way?

\begin{question}
Is $\Delta^0_n\text{-}\DRT^2_2  \equiv_{\mathrm{sW}}  D^n_2$?
\end{question}

These are just a few of the many questions that remain 
concerning these principles.

\bibliography{Dual}

\begin{thebibliography}{10}

\bibitem{AK}
C.~J. Ash and J.~F. Knight.
\newblock {\em Computable structures and the hyperarithmetic hierarchy}.
\newblock Studies in Logic and the Foundations of Mathematics. Elsevier,
  Amsterdam, 2000.

\bibitem{ADMSW}
Eric~P. Astor, Damir Dzhafarov, Antonio Montalb\'{a}n, Reed Solomon, and
  Linda~Brown Westrick.
\newblock The determined property of {B}aire in reverse math.
\newblock {\em J. Symb. Log.}, 85(1):166--198, 2020.

\bibitem{bhs}
Andreas~R. Blass, Jeffry~L. Hirst, and Stephen~G. Simpson.
\newblock Logical analysis of some theorems of combinatorics and topological
  dynamics.
\newblock In {\em Logic and combinatorics ({A}rcata, {C}alif., 1985)},
  volume~65 of {\em Contemp. Math.}, pages 125--156. Amer. Math. Soc.,
  Providence, RI, 1987.

\bibitem{cs}
Timothy~J. Carlson and Stephen~G. Simpson.
\newblock A dual form of {R}amsey's theorem.
\newblock {\em Adv. in Math.}, 53(3):265--290, 1984.

\bibitem{ChongLemppYang}
C.~T. Chong, Steffen Lempp, and Yue Yang.
\newblock On the role of the collection principle for {$\Sigma^0_2$}-formulas
  in second-order reverse mathematics.
\newblock {\em Proc. Amer. Math. Soc.}, 138(3):1093--1100, 2010.

\bibitem{Dzhafarov-2015ta}
Damir~D. Dzhafarov.
\newblock Strong reductions between combinatorial principles.
\newblock {\em Journal of Symbolic Logic}, 81(4):1405--1431, 2016.

\bibitem{e}
Julia~Christina Erhard.
\newblock {\em The Carlson-Simpson Lemma in Reverse Mathematics}.
\newblock PhD thesis, UC Berkeley: Mathematics, 2013.

\bibitem{GM}
Noam Greenberg and Antonio Montalb\'{a}n.
\newblock Ranked structures and arithmetic transfinite recursion.
\newblock {\em Transactions of the AMS}, 360:1265--1307, 2008.

\bibitem{HJ-2015ta}
Denis~R. Hirschfeldt and Carl~G. Jockusch, Jr.
\newblock On notions of computability theoretic reduction between {$\Pi^1_2$}
  principles.
\newblock {\em Journal of Mathematical Logic}, 16, 2016.

\bibitem{Jockusch1968}
Carl~G. Jockusch, Jr.
\newblock Uniformly introreducible sets.
\newblock {\em J. Symbolic Logic}, 33:521--536, 1968.

\bibitem{LiuMoninPatey}
Lu~Liu, Benoit Monin, and Ludovic Patey.
\newblock A computable analysis of variable words theorems.
\newblock {\em Proc. Amer. Math. Soc.}, 147(2):823--834, 2019.

\bibitem{ms}
Joseph~S. Miller and Reed Solomon.
\newblock Effectiveness for infinite variable words and the dual {R}amsey
  theorem.
\newblock {\em Arch. Math. Logic}, 43(4):543--555, 2004.

\bibitem{patey}
Ludovic Patey.
\newblock Private communication, 2014.

\bibitem{pv}
Hans~J{\"u}rgen Pr{\"o}mel and Bernd Voigt.
\newblock Baire sets of {$k$}-parameter words are {R}amsey.
\newblock {\em Trans. Amer. Math. Soc.}, 291(1):189--201, 1985.

\bibitem{Rogers_book}
Hartley Rogers, Jr.
\newblock {\em Theory of recursive functions and effective computability}.
\newblock MIT Press, Cambridge, MA, second edition, 1987.

\bibitem{sacks}
Gerald~E. Sacks.
\newblock {\em Higher recursion theory}.
\newblock Perspectives in Mathematical Logic. Springer-Verlag, Berlin, 1990.

\bibitem{sosa}
Stephen~G. Simpson.
\newblock {\em Subsystems of second order arithmetic}.
\newblock Perspectives in Logic. Cambridge University Press, Cambridge;
  Association for Symbolic Logic, Poughkeepsie, NY, second edition, 2009.

\bibitem{slaman}
Theodore~A. Slaman.
\newblock A note on the {D}ual {R}amsey {T}heorem.
\newblock 4 pages, unpublished, 1997.

\bibitem{TodorcevicRamsey}
Stevo Todorcevic.
\newblock {\em Introduction to {R}amsey spaces}, volume 174 of {\em Annals of
  Mathematics Studies}.
\newblock Princeton University Press, Princeton, NJ, 2010.

\end{thebibliography}
\bibliographystyle{plain}

\end{document}